\documentclass{amsart}
\usepackage{amssymb}
\usepackage{amsfonts}
\usepackage{amsthm}
\usepackage{amsmath}
\usepackage{mathtools}
\usepackage{bbm}
\usepackage{bussproofs}
\usepackage{mathrsfs}
\usepackage{hyperref}
\usepackage{tikz, tikz-cd}
\usepackage{fullpage}
\usepackage[all,2cell]{xy}
\xyoption{2cell}
\usepackage{ stmaryrd } 
\usepackage{todonotes}

\tikzset{%
	symbol/.style={%
		draw=none,
		every to/.append style={%
			edge node={node [sloped, allow upside down, auto=false]{$#1$}}}
	}
}

\usetikzlibrary{matrix,arrows}

\newtheorem{Theorem}{Theorem}

\newtheorem{proposition}[Theorem]{Proposition}
\newtheorem{lemma}[Theorem]{Lemma}
\newtheorem{corollary}[Theorem]{Corollary}

\theoremstyle{definition}

\newtheorem{remark}[Theorem]{Remark}
\newtheorem{definition}[Theorem]{Definition}

\DeclareMathOperator{\Abb}{\mathbb{A}}

\DeclareMathOperator{\Fbb}{\mathbb{F}}
\DeclareMathOperator{\Gbb}{\mathbb{G}}

\DeclareMathOperator{\Acal}{\mathcal{A}}

\DeclareMathOperator{\Gcal}{\mathcal{G}}

\DeclareMathOperator{\Ocal}{\mathcal{O}}

\DeclareMathOperator{\mfrak}{\mathfrak{m}}

\DeclareMathOperator{\pfrak}{\mathfrak{p}}

\DeclareMathOperator{\qfrak}{\mathfrak{q}}

\DeclareMathOperator{\id}{id}

\DeclareMathOperator{\Ker}{Ker}

\DeclareMathOperator{\SSS}{\mathbb{S}}
\DeclareMathOperator{\Set}{\mathbf{Set}}

\DeclareMathOperator{\C}{\mathbb{C}}
\DeclareMathOperator{\Z}{\mathbb{Z}}
\DeclareMathOperator{\N}{\mathbb{N}}

\DeclareMathOperator{\Gal}{Gal}

\DeclareMathOperator{\Ab}{\mathbf{Ab}}
\DeclareMathOperator{\Grp}{\mathbf{Grp}}
\DeclareMathOperator{\Q}{\mathbb{Q}}

\DeclareMathOperator{\Ext}{Ext}
\DeclareMathOperator{\cd}{cd}

\DeclareMathOperator{\CalO}{\mathcal{O}}

\DeclareMathOperator{\Spec}{Spec}

\DeclareMathOperator{\Cat}{\mathbf{Cat}}

\DeclareMathOperator{\Top}{\mathbf{Top}}

\DeclareMathOperator{\Coker}{Coker}

\DeclareMathOperator{\Fp}{\mathbb{F}_{\mathnormal{p}}}

\DeclareMathOperator{\A}{\mathbb{A}}

\DeclareMathOperator{\op}{op}

\DeclareMathOperator{\Iner}{I}
\DeclareMathOperator{\G}{\mathbb{G}}
\DeclareMathOperator{\QAGrp}{\mathbf{QAlgGrp}}
\DeclareMathOperator{\PAGrp}{\mathbf{ProAlgGrp}}
\DeclareMathOperator{\GL}{GL}
\DeclareMathOperator{\ProfAGrp}{\mathbf{ProAb}}

\DeclareMathOperator{\Frob}{Frob}

\let\epsilon\varepsilon

\newcommand{\pullbackcorner}[1][dr]{\save*!/#1-1.2pc/#1:(-1,1)@^{|-}\restore}

\newcommand{\Galf}[0]{Galf}

\UseTwocells

\numberwithin{Theorem}{section}
\numberwithin{equation}{section}

\title{Serre-Hazewinkel Local Class Field Theory and a Geometric Proof of the Local Langlands Correspondence for $\GL(1)$}
\author{Geoff Vooys}
\email{gmvooys@ucalgary.ca}
\date{\today}

\subjclass{Primary 11S31; Secondary 14L40, 11F70}
\begin{document}

\maketitle
\begin{abstract}
In this expository paper we provide a geometric proof of the local Langlands Correspondence for the groups $\operatorname{GL}_{1}$ defined over $p$-adic fields $K$. We do this by redeveloping the theory of proalgebraic groups and use this to derive local class field theory in the style of Serre and Hazewinkel. In particular, we show that the local class field theory of Serre and Hazewinkel is valid for both equal characteristic and mixed characteristic ultrametric local fields. Finally, we use this to prove an equivalence of the categories of smooth representations of $K^{\ast}$ with continuous representations of $W_K^{\text{Ab}}$ in order to deduce the Local Langlands Correspondence for $\operatorname{GL}_{1,K}$.
\end{abstract}
\section{Introduction}
Local Class Field Theory is a fundamental area of number theory, representation theory, Galois cohomology, and arithmetic geometry. In fact, the Local Langlands correspondence for $p$-adic $\GL(1)$ is essentially given by local class field theory for $p$-adic fields (cf.\@ Theorem \ref{Cor: LLC for padic GL1}). In this article we seek to give a basic introduction to local class field theory; we will do so from a different perspective than the ultrametric analytic and algebraic number theoretic proofs from which local class field theory is usually derived (cf.\@ \cite{NeukirchANT}, for example, on such a development of local class field theory). Instead, we will deduce local class field theory from the theory of proalgebraic groups of Serre (cf.\@ \cite{SerreProAlgGrp} and \cite{SerreCorps}) and the further development of the theory of proalgebraic groups of Hazewinkel (as summarized in \@ \cite{DemazureGabriel}). In presenting this, we follow also the combined exposition of \cite{SuzYosh} to some degree.

The biggest reason we derive local class field theory in this way is to provide a schematic background to the theory. By providing this geometry to the problem, we can also introduce geometric methods to the study of local class field theory and show how the resulting field theory is induced from the study of (pro)algebraic groups; this will also have the benefit of explaining that the Serre-Hazewinkel proalgebraic approach to local class field theory is valid not just for equal characteristic local fields, but also in the mixed characteristic (i.e., $p$-adic) case as well. The final reason we derive local class field theory in this way is to  introduce the theory of proalgebraic groups to a modern audience, as proalgebraic groups have seen use and development in artimetic geometry, differential Galois Theory, model theory, and representation theory; cf.\@ \cite{FreeProalg}, \cite{ModelTheoryProAlg}, \cite{DiffGaloisRatFun}, and \cite{MasoudTravis} for examples and instances of these applications of proalgberaic groups.

We will now fix some notation. Let $K$ be a complete discrete valuation field (with valuation $v$) of characteristic 0, let $R$ be its valution ring, i.e., its local ring of integers, and let $k = R/\mathfrak{m}$ be its (perfect) residue field of characteristic $p > 0$. Set $K^{{\text{Ab}}}$ be the maximal Abelian extension of $K$. Moreover, we fix $U_K := \lbrace x \in K \; | \; v(x) = 0 \rbrace$ as the group of valuative units of $K$. Classical local class field theory states that if $k = \mathbb{F}_{q}$ for $q = p^n$ then there is a canonical group homomorphism
\[
K^{\ast} \to \Gal(K^{{\text{Ab}}}/K)
\]
that induces a commuting diagram
\[
\xymatrix{
0  \ar[r] & U_k \ar[r] \ar[d]_{\cong} & K^{\ast} \ar[r] \ar[d] & \Z\ar[r] \ar[d] & 0 \\
0 \ar[r] & \Iner(K^{{\text{Ab}}}/K) \ar[r] & \Gal(K^{{\text{Ab}}}/K) \ar[r] & \Gal(\overline{k}/k) \ar[r] & 0
}
\]
with exact rows in $\Ab$, where $\Iner(K^{\text{Ab}}/K)$ is the inertia group  of $K^{\text{Ab}}/K$; it is also worth remarking that there is an isomorphism
\[
\Iner(K^{\text{Ab}}/K) \cong \Iner(\overline{K}/K)^{\text{Ab}}.
\] 
Serre later extended this persepective to the case in which $k$ is any algebraically closed field of characteristic $p$ (such as $k = \overline{\Fbb_p(t)}$, for instance), and then Hazewinkel extended it further to the case in which $k$ is a perfect, but not necessarily algebraically closed, field of characteristic $p$. Suzuki and Yoshida then further generalized this to the case in which $k$ was a general perfect field, and we largely follow the three papers above in presenting this material. 

The idea behind this approach to LCFT is to make a fundamental group which relates the representation-theoretic side of things to the geometric side of things. In order to make this precise we must therefore construct a geometric framework in which we can introduce the coverings and fundamental groups. In particular, we will need to work with a certain class of group schemes over the residue field $k$ that will allow us to work appropriately. However, the groups we will work with will be infinite dimensional group schemes over $k$, so we will need to work with some sort of fundamental groups which are not, for instance, finite {\'e}tale. In particular, we will need a way to realize the group of units $U_K$ as a group scheme over $k$. However, the scheme we will attach to $U_K$ is what is called a proalgebraic group and is a limit of algebraic groups in a certain sense, so we can restrict our attention to a nice (Abelian) category of commutative group schemes.

\section{Background on QuasiAlgebraic Groups, Proalgebraic Groups, and The Proalgebraic Fundamental Group}
We begin by fixing a field $k$ of characteristic $p$. Because $k$ has characteristic $p$, there is a Frobenius endomorphism $\operatorname{Fr}(p):k \to k$ given by $x \mapsto x^p$. For any scheme $X$ over $\Spec k$, this characteristic $p$ structure implies that every structure ring $\CalO_X(U)$ of $X$ a $k$-algebra as well. This in turn implies that there is a Frobenius endomorphism $\operatorname{Fr}:X \to X$ induced by taking $x \mapsto x^p$ in each ring $\CalO_X(U)$ and acting as the identity on $\lvert X\rvert$; that this is compatible with the restriction maps is trivial to check. This determines a morphism
\[
\operatorname{Fr}:X \to X
\]
called the absolute Frobenius of $X$. In particular, if $k$ be a field of characteristic $p > 0$ a scheme $X$ over $\Spec k$ is said to be {\em perfect} whenever the absolute Frobenius $\operatorname{Fr}:X \to X$ is an automorphism.

One issue that arises when working with arbitrary schemes over a finite field $\Fbb_q$ is that they are not always perfect. For instance, for any algebraic field extension $L/\Fp$, the group scheme $\G_{a,L} = \Abb^{1}_{L}$ is not perfect because $x$ has no $p^{\text{th}}$ root in $L[x]$. Perhaps more worryingly, even some reductive groups, like $\G_{m,L}$ need not be perfect, as neither $x$ {\em nor} $y$ has a $p^{\text{th}}$ root in $L[x,y]/(xy-1)$. Thus we will proceed by introducing a category of group schemes over the finite field $k$ which arise by taking the perfection of commutative group varieties over $k$ in the standard pro-way, i.e., by adding all $p^{n}$-th roots of elements of the structure sheaves.

Begin by fixing a group variety $G$ over $\Spec k$. Now, for every $n \in \N$, define the scheme $G^{(n)}$ by declaring that as a locally ringed space
\[
\lvert G^{(n)} \rvert := \lvert G \rvert,
\]
and replacing the sheaf $\CalO_{G}$ by the sheaf induced by pushing $\CalO_G$ forward and then taking the pullback of the composite
\[
\xymatrix{
G \ar[r] \ar[dr]_{f} & \Spec k \ar[d]^{\Spec \operatorname{Fr}(p^n)} \\
 & \Spec k	
}
\]
where $\operatorname{Fr}(p^n)$ is the Frobenius automorphism of $\Spec k$ at $p^n$. Explicitly, $\CalO_{G^{(n)}} = f^{\ast}f_{\ast}\CalO_G$. There are then canonical morphisms
\[
G^{(n+k)} \to G^{(n)}
\]
for all $n, k \in \N$ that correspond to producing higher powers of the Frobenius endomorphism; note that $G^{(n)}$ is a group scheme whenever $G$ is a group variety. We then define the scheme $G^{(\infty)}$ via
\[
G^{(\infty)} := \lim_{\longleftarrow} G^{(n)}.
\]
The following lemma characterizes some of the important properties of $G^{(\infty)}$.
\begin{lemma}\label{Quasialgebraic group  is group scheme}
Let $G$ be a group variety over $\Spec k$ and let $G^{(\infty)}$ be the scheme constructed above. Then:
\begin{enumerate}
	\item The scheme $G^{(\infty)}$ is a group scheme;
	\item The scheme $G^{(\infty)}$ is perfect.
\end{enumerate}
\end{lemma}
\begin{proof}[Sketch of Proof]
(1): This follows directly from the fact that limits commute. In particular, if $\mu_n:G^{(n)} \times G^{(n)} \to G^{(n)}$ is the multiplication of each $G^{(n)}$, then we see that
\[
\lim_{\longleftarrow}(G^{(n)} \times G^{(n)}) \cong \lim_{\longleftarrow}G^{(n)} \times \lim_{\longleftarrow}G^{(n)} = G^{(\infty)} \times G^{(\infty)}.
\]
It follows that $\tilde{\mu}:G^{(\infty)} \times G^{(\infty)} \to G^{(\infty)}$ is given by
\[
\mu_{\infty} = \lim_{\longleftarrow} \mu_{n}.
\]
Similarly, $1_{G^{(\infty)}} = \lim 1_{G^{(n)}}$ and $\nu_{\infty} = \lim \nu_n$. The verification that these maps make $G^{(\infty)}$ into a group scheme then follows from the fact that each object $G^{(n)}$ is a group with operations $\mu_n, \nu_n, 1_{G^{(n)}}$.

(2): By construction, it follows that the induced Frobenius morphism on $G^{(\infty)}$ is an automorphism. In particular, this follows from the construction induced by colimiting along the sheaves after precomposing with all possible Frobenius morphisms.
\end{proof}
\begin{definition}[\cite{SerreProAlgGrp}]
A {\em quasi-algebraic group} is a {commutative} group scheme $H$ over $k$ for which there exists a {commutative} group variety $G$ over $k$ such that $H \cong G^{(\infty)}$. The category of all such groups and their group homomorphisms will be denoted $\QAGrp_{/k}$. If the field $k$ is clear from context, we simply write $\QAGrp$ instead.
\end{definition}
\begin{definition}
The pro-category of $\QAGrp$ is the category of all {\em pro-algebraic groups} over $k$, and will be denoted $\PAGrp_{/k}$. As above, if the field $k$ is clear from context, we write $\PAGrp$.
\end{definition}
\begin{remark}
A group scheme $G$ over $\Spec k$ that is an object of $\PAGrp_{/k}$ is called a {\em proalgebraic group}. More explicitly, let us see what information determines a proalgebraic group. Since a proalgebraic group is a pro-object of $\QAGrp_{/k}$, the following data determines such a group (cf.\@ \cite{SerreProAlgGrp} and \cite{SerreCorps}):
\begin{enumerate}
	\item A group scheme $G$ together with a collection of subgroup schemes $S$ such that for all $H \in S$, the quotient scheme $G/H$ exists and is a quasialgebraic group over $k$;
	\item If $H \in S$ and if $H^{\prime} \in S$, then $H \cap H^{\prime} \in S$. Moreover, the maps
	\[
	\frac{G}{H \cap H^{\prime}} \to \frac{G}{H}
	\]
	and
	\[
	\frac{G}{H \cap H^{\prime}} \to \frac{G}{H^{\prime}}
	\]
	are both morphisms of quasialgebraic groups;
	\item If $H \in S$ and if $K \in S$ with $H \hookrightarrow K$, then $K$ arises as the inverse image of some subgroup of $G/H$;
	\item There is an isomorphism of group schemes
	\[
	G \cong \lim_{\substack{\longleftarrow \\ H \in S}} \frac{G}{H}.
	\]
\end{enumerate}
Assume that $(G,S)$ and $(H,T)$ are pairs of data described above and regard $G$ and $H$ as proalgebraic groups. If there is a morphism $f \in \PAGrp(G,T)$, then $f$ gives a morphism $f:(G,S) \to (H,T)$ in the relevant sense, i.e., in particular, $f^{-1}(T) \subseteq S$. More explicitly, for any $K \in T$, $f^{-1}(K) \in S$. Moreover, the induced map $G/f^{-1}(K) \to H/K$ is also a morphism of quasialgebraic groups.
\end{remark}
The reason we restrict our attention to commutative quasialgebraic groups, and hence to commutative proalgebraic groups, is because this makes $\PAGrp_{/k}$ into an Abelian category. It may be worth extending this to noncommutative proalgebraic groups, as the construction of $G^{(\infty)}$ did not rely on commutativity in any way, but for the moment and for recovering Serre-Hazewinkel LCFT it is not necessary.
\begin{proposition}[\cite{SerreProAlgGrp}, Proposition 2.6]
The category $\PAGrp_{/k}$ is an Abelian category that satisfies the axiom Ab.5 and has a projective cogenerator.
\end{proposition}

In order to give $U_K$ the structure of a proalgebraic group, we will need to see it as a limit of the groups $\Gbb_{m,R/\mfrak^n}^{(\infty)}$. This will connect with our $p$-adic intuition of realizing $U_K$ as the limit of the groups $U_K^{(n)}$ of units in an $n$-infinitesimal neighborhood of the special fibre of $\Spec R$. However, to make this argument we need a way of viewing the Artinian rings $R/\mfrak^n$ as varieities over $k$.

We now must show how to view each of the Artinian local rings $R/\mathfrak{m}^n$ as varieties over $k$; the argument we give here is essentially due to Marvin J.\@ Greenberg, but is presented in \cite{SerreCorps}. First, recall that the ring of truncated Witt vectors over $k$, $W_{n}(k)$, may be viewed as a $k$-variety by seeing $\mathbf{W}_n$ as the group scheme with underlying scheme $\A_{k}^{n+1}$ whose addition, negation, multiplication, unit, and zero are all given by the standard truncated Witt polynomials; see \cite{SerreLocalFields}, \cite{Borger1}, for instance, for details on this process. It also follows from basic arguments in the theory of Witt vectors that $W_n(k)$ is an Artinian local ring with $W_n(k)/\mathfrak{n} \cong k$, where $\mathfrak{n}$ is the maximal ideal of $W_n(k)$. We can then find a sufficiently large $n \in \N$ so that we have a ring homomorphism $W_n(k) \to R/\mathfrak{m}^n$ which realizes $R/\mathfrak{m}^n$ as a finite-type $W_n(k)$-module. In particular, as a $W_n(k)$-module, we have an isomorphism
\[
\frac{R}{\mathfrak{m}^n} \cong \bigoplus_{i = 0}^{\ell} W_{n_i}(k)
\]
for $0 \leq n_i \leq n$  which  also respects the ring structure. Since each of the $W_{n_i}(k)$ has the structure of a $k$-variety, Transport-of-Structure gives a $k$-variety structure to $R/\mathfrak{m}^n$ in such a way that does not depend on the isomorphism
\[
\frac{R}{\mathfrak{m}^{n}} \cong \prod_{i=0}^{\ell} W_{n_i}(k)
\]
chosen; this is the way in which we will view $R/\mathfrak{m}^n$ as a $k$-variety throughout this chapter.

Let us now proceed to show how to give the group $U_K$ the structure of a proalgebraic group. Define the quasialgebraic group
\[
\G_{m,k}^{(\infty)} := \lim_{\longleftarrow} \G_{m,k}^{(n)};
\]
we will use this group, together with some of its base changes to the varieties $\G_{m,R/\mathfrak{m}^n}$, to construct the group scheme $\mathbf{U}_K$ which coincides with $U_K$ on global points. We need to now understand the structure of $U_K$ as a group in order to make it into a proalgebraic group. In particular, we would like to see how to view $U_K$ as a $p$-adic type of limit. So, to do this, we first define the (multiplicative) subgroups
\[
U_{K}^{(n)} := 1 + \mathfrak{m}^n \leq U_K
\]
for all $n \in \N$. Elementary arguments in the theory of local fields then show that there are two important isomorphisms of groups, the first connecting each group of units in $R/\mathfrak{m}$ to $U_K/U_{K}^{(n)}$
\[
\left(\frac{R}{\mathfrak{m}^n}\right)^{\ast} \cong \frac{U_K}{U_K^{(n)}},
\]
and the second connecting $U_K$ to the limit against the quotients of $U_K$ over the groups of principal units
\[
U_K \cong \lim_{\longleftarrow} \frac{U_K}{U_K^{(n)}} \cong \lim_{\longleftarrow} \left(\frac{R}{\mathfrak{m}^n}\right)^{\ast}.
\]
Define the proalgebraic group $\mathbf{U}_K$ over $k$ by first considering the groups $\G_{m,R/\mathfrak{m}^{n}}^{(\infty)}$, where $R/\mathfrak{m}^n$ has the variety structure considered above, and then taking
\[
\mathbf{U}_K := \lim_{\longleftarrow} \G_{m,R/\mathfrak{m}^n}^{(\infty)}.
\]
This leads to the following proposition:
\begin{proposition}\label{Prop UK is UK}
Consider the scheme $\mathbf{U}_K$ defined above. Then the following hold:
\begin{enumerate}
	\item $\mathbf{U}_K$ is a commutative proalgebraic group;
	\item $\mathbf{U}_K(k) \cong U_K$;
	\item For every $n \in \N$, there is a closed subgroup $\mathbf{U}_K^{(n)}$ of $\mathbf{U}_K$ associated to $U_K^{(n)}$ in that $\mathbf{U}_K^{(n)}(k) = U_K^{(n)}$;
\end{enumerate}
\end{proposition}
\begin{proof}
(1): This holds from the fact that for each $n \in \N$, $\G_{m,R/\mathfrak{m}^n}^{(\infty)}$ is a commutative quasialgebraic group.

(2): For this we calculate that
\[
\mathbf{U}_K(k) = \lim_{\longleftarrow}\left(\G_{m,R/\mathfrak{m}^{n}}^{(\infty)}(k)\right) \cong \lim_{\longleftarrow}\left(\frac{R}{\mathfrak{m}^n}\right)^{\ast} \cong \lim_{\longleftarrow}\frac{U_K}{U_K^{(n)}} \cong U_K.
\]

(3): The scheme $\mathbf{U}_{K}^{(n)}$ associated to $U_{K}^{(n)}$ comes from the fact that the sheaf $1 + \mathscr{M}^n$ is a subgroup sheaf of $\G_{m,R}$ and then doing the same construction as $\mathbf{U}_K$ mutatis mutndis with $U_K/U_{K}^{(m)}$ replaced by $U_K^{(n)}/U_K^{(m)}$ for $m \geq n$. 

\end{proof}

\begin{remark}
From $(3)$, it follows that
\[
\frac{\mathbf{U}_K}{\mathbf{U}_K^{(1)}} \cong \G_{m,k}^{(\infty)}
\]
by noting that
\[
\frac{\mathbf{U}_K}{\mathbf{U}_K^{(1)}} \cong \lim_{\substack{\longleftarrow \\ n \in \N}} \frac{\G_{m,R/\mathfrak{m}^{n}}^{(\infty)}}{\G_{m,R/\mathfrak{m}^{n+1}}^{(\infty)}} \cong \lim_{\substack{\longleftarrow \\ n \in \N}} \G_{m,k}^{(\infty)} = \G_{m,k}^{(\infty)}.
\]
\end{remark}

Using the general theory of Witt vectors and strict $p$-rings (cf.\@ \cite{SerreLocalFields} for a classical perspective), since the group scheme $\G_{m,k}^{(\infty)}$ is perfect it lifts via Witt vectors to a scheme over $\Spec R$. In particular, since the quotient $q_{\mathbf{U}_K}:\mathbf{U}_K^{(1)} \to \mathbf{U}_K/\mathbf{U}_K^{(1)}$ is equivalent to reducing the scheme modulo the maximal ideal $\mathfrak{m}$ of $R$, the Teichm{\"u}ller section $\tau:k \to R$ induces a section of $q_{\mathbf{U}_K^{(1)}}$, after post-composing with the isomorphism between the quotient and $\G_{m,k}^{(\infty)}$. This implies that the diagram
\[
\xymatrix{
	\G_{m,k}^{(\infty)} \ar[r] \ar@{=}[d] & \mathbf{U}_K \ar[d]^{q_{\mathbf{U}_{K}^{(1)}}} \\
	\G_{m,k}^{(\infty)} & \mathbf{U}_K/\mathbf{U}_{K}^{(1)} \ar[l]^{\cong}
}
\]
commutes, which in turn gives us the proposition below after appealing to the (Short) Five Lemma.

\begin{proposition}
The Teichm{\"u}ller section $\G_{m,k}^{(\infty)} \to \mathbf{U}_K$ induces a splitting of proalgebraic groups
\[
\mathbf{U}_K \cong \G_{m,k}^{(\infty)} \times \mathbf{U}_K/\mathbf{U}_K^{(1)},
\]
where the product is taken in $\PAGrp_{/k}$.
\end{proposition}
\begin{proposition}
	The proalgebraic group $\mathbf{U}_K$ is connected.
\end{proposition}
\begin{proof}
Since
\[
\mathbf{U}_K \cong \lim_{\longleftarrow} \G_{m,R/\mathfrak{m}^n}^{(\infty)} \cong \lim_{\longleftarrow} \G_{m,R/\mathfrak{m}^n}^{(\infty)},
\]
and since each of the groups $\G_{m,R/\mathfrak{m}^n}^{(\infty)}$ is itself connected (as the general construction of the schemes $\G_{m,-}$ are connected and the underlying scheme of each $R/\mathfrak{m}^n$ is some connected subscheme of some $\G_{a,k}^m$ for $m $ sufficiently large), the result follows.
\end{proof}

\begin{definition}
Let $(G,S)$ be a proalgebraic group over $k$. Then we say that $G$ is {\em connected} (respectively that $G$ is {\em zero dimensional}) if for every subgroup $H \in S$, the group scheme $G/H$ is connected (respectively finite).
\end{definition}

It is worth observing that for any proalgebraic group $G$ there is a maximal closed subgroup $G_0$ of $G$ which is connected (and is defined by the maximal profinite quotient of $G$). In particular, this assignment is functorial, and so is taking quotients by this maximal connected subgroup. Thus there is a functor $\pi_0:\PAGrp_{/k} \to \ProfAGrp$, where $\ProfAGrp$ is the Abelian category of Abelian profinite groups. Note that $\pi_0(G)$ is given by the quotient group
\[
\pi_0(G) := \frac{G}{G^0},
\]
where $G^{0}$ is the neutral component of $G$, i.e., the connected component of the identity of $G$.
Note that since we have quotiented out the maximal connected subgroup  of $G$, $\pi_0(G)$ is a zero-dimensional proalgebraic group. Similarly, if $G$ is a zero-dimensional proalgebraic group, then $G$ may be identified as the perfect constant group scheme on $\pi_0(G)$, which is in particular an Abelian profinite group.

\begin{definition}[\cite{SerreProAlgGrp} and \cite{SuzYosh}]
The group $\pi_0^{k}(G)$ is called the {\em group of connected components of $G$} or the {\em zero-th homotopy group of $G$}.
\end{definition}

We would now like to understand this $0$-th homotopy group more in depth.  Let us begin by studying the categorical properties of the functor $\pi_0$. Observe that since $\pi_0$ is a functor defined by taking a quotient, which is in particular a type of colimit, it follows that $\pi_0$ is right exact by the commutativity of colimits with other colimits. Moreover, the functors $\pi_0$ are not left exact, although they do commute with {\em filtered} limits; thus there should be an analogous long exact homotopy sequence of groups induced by taking left derived functors. However, while this may be formally done via the theory of derived categories (as the functor $\pi_0$ is additive), we should try to understand what hom-spaces the functors $L_i\pi_0$ will classify, in analogy to the fact that for the topological fundamental groups $\pi_n^{\Top}(X,x) \cong \Top_{\ast}((\SSS^{n},1),(X,x))$.

To see what the $L_i\pi_0$ classify, we will follow \cite{SerreProAlgGrp} and use some basic homological algebra of derived functors. In particular, we will understand that the $\Ext$ functors correspond to right-derived functors of hom-functors. Let $\mathbf{N}$ be a zero-dimensional proalgebraic group over $k$ and let $N$ be the profinite group associated to $N$; note this follows because each quotient $\mathbf{N}/\mathbf{H}$ of $\mathbf{N}$ by a specified closed subscheme $\mathbf{H}$ is finite, and that $N = \mathbf{N}(\overline{k})$. Then we get the following lemma:
\begin{lemma}[\cite{SerreProAlgGrp}]\label{Lemma: Factorizing G,N to pi0G, N}
The functor
\[
\PAGrp(-,\mathbf{N}):\PAGrp^{\op} \to \ProfAGrp
\]
factors as
\[
\xymatrix{
\PAGrp^{\op} \ar[dr]_{\pi_0} \ar[rr]^{\PAGrp(-,\mathbf{N})} & & \ProfAGrp \\
 & \ProfAGrp^{\op} \ar[ur]_{\ProfAGrp(-,N)}
}
\]
in the category $\Cat$ of categories. In particular,
\[
\PAGrp(G,\mathbf{N}) \cong \ProfAGrp(\pi_0(G), N).
\]
\end{lemma}
\begin{proof}[Sketch]
Since $\mathbf{N}$ is zero-dimensional, morphisms $G \to \mathbf{N}$ lie in bijection with morphisms of profinite groups $G(\overline{k}) \to \mathbf{N}(k)$. Moreover, since $N$ is profinite, it is totally disconnected; as such, it is constant on the neutral component of $G$. Thus any map $G \to \mathbf{N}$ uniquely factors through $\pi_0(G)$, which was to be shown.
\end{proof}

The above lemma tells us that the $0$-th fundamental group classifies maps with profinite codomain in a very relevant sense. However, since a profinite group takes the form
\[
N \cong \lim_{\longleftarrow} N/H,
\]
where each quotient $N/H$ is a finite group, we have that
\[
\ProfAGrp(\pi_0(G),N) \cong \ProfAGrp\left(\pi_0(G),\lim_{\longleftarrow}\frac{N}{H}\right) \cong \lim_{\longleftarrow}\ProfAGrp\left(\pi_0(G), \frac{N}{H}\right)
\]
and so $\pi_0^{k}(G)$ determines maps out of $G$ with locally finite image, which is  to say locally finite isogenies.

We now only need one more technical result about the category $\PAGrp_{/k}$, which involves projective resolutions. In particular, we will cite, but not prove (as it is technical but not as informative as one may hope), a proposition that shows projective objects in $\PAGrp_{/k}$ may be taken exactly from projective objects in the full subctegory $\PAGrp_{/k}^{0}$ of zero-dimensional proalgebraic groups. This will allow us  to deduce that when one considers zero-dimensional proalgebraic groups, it suffices to take projective resolutions whose objects are all zero-dimensional, and then deduce further with the help of Lemma \ref{Lemma: Factorizing G,N to pi0G, N} that there is a spectral sequence for $\Ext$ groups involving $\pi_0$.

\begin{proposition}[Proposition 5, Section 4 of \cite{SerreProAlgGrp}]
Let $\G$ be a zero-dimensional proalgebraic group over $k$ and let $G$ be the associated profinite group to $\G$. Then the following are equivalent:
\begin{enumerate}
	\item $\G$ is a projective object in $\PAGrp_{/k}^{0}$;
	\item $\G$ is a projective object in $\PAGrp_{/k}$;
	\item $G$ is isomorphic to a product of the groups $\Z_{\ell}$ for integer primes $\ell$;
	\item $G$ is torsion-free;
	\item The group $\Ab(G,\Q/\Z)$ is divisible.
\end{enumerate}
\end{proposition}
\begin{proof}[Sketch]
The general theory of torsion profinite groups shows that $(3), (4),$ and $(5)$ are all equivalent when one realizes that
\[
\Ab(G,\Q/\Z) \cong \Ab\left(G,\lim_{\longrightarrow} \frac{\Z}{n\Z}\right) \cong \lim_{\longrightarrow} \Ab\left(G,\frac{\Z}{n\Z}\right).
\]
The fact that $(2) \implies (1)$ is trivial, while the equivalence of $(1)$ and $(4)$ may be shown directly from the fact that $\G$ is associated to $G$ and then using the projective condition to imply that you cannot have the desired lifts if $G$ has nontrivial torsion. Finally, $(3) \implies (2)$ is Proposition 4 of Section 4 of \cite{SerreProAlgGrp}.
\end{proof}
\begin{corollary}
If $G$ is a zero-dimensional proalgebraic group, any projective resolution of $G$ in $\PAGrp_{/k}^{0}$ is also a projective resolution of $G$ in $\PAGrp_{/k}$.
\end{corollary}

\begin{definition}
If $G$ is a proalgebraic group over $k$, then we define the {\em higher homotopy groups of $G$} to be the left derived functors of $\pi_{0}$ of $G$, i.e., for all $n \in \N$,
\[
\pi_{n}(G) := L_{n}\pi_{0}(G).
\]
The group $\pi_1(G)$ is called the {\em fundamental group of $G$}.
\end{definition}

Note that from this definition, it follows that if
\[
\xymatrix{
0 \ar[r] & G \ar[r] & H \ar[r] & K \ar[r] & 0	
}
\]
is a short exact sequence of proalgebraic groups, then there is an induced long exact sequence of proalgebraic groups:
\[
\xymatrix{
 \cdots \ar[r] & \pi_{n}(G) \ar[r] & \pi_{n}(H) \ar[r] & \pi_{n}(K) \ar[r]^-{\delta_{n-1}} & \pi_{n-1}(G) \ar[r] & \cdots \ar[r] & \pi_{2}(G) \ar[d]^{\delta_1} \\
 &\pi_{0}(K)  & \pi_{0}(H)\ar[l] & \pi_{0}(G)\ar[l] & \pi_1(K) \ar[l]^{\delta_0} & \pi_1(H)\ar[l] & \pi_1(G) \ar[l]	
}
\]
We will use this to finally see the spectral sequence between $\Ext$ groups and $\pi_i^k$ groups, which will be particularly important to the proof of Theorem \ref{Theorem: LCFT of Hazewinkel}.
\begin{Theorem}\label{Theorem: Spectral Sequence for fundamental groups}
Let $G$ be a proalgebraic group over $k$ and let $\mathbf{N}$ be a zero-dimensional proalgebraic group  over $k$ with associated profinite group $N$. For any $p,q \in \N$, there is a spectral sequence
\[
\Ext^{p}(\pi_q^{k}(G),\mathbf{N}) \Rightarrow \Ext^{p+q}(G,\mathbf{N}).
\]
In particular, if $G$ is connected there is an isomoprhism
\[
\Ext^{1}(G,\mathbf{N}) \cong \ProfAGrp(\pi_1(G),N).
\]
\end{Theorem}

This will tell us, as usual, about the connectivity/hole based structure of proalgebraic groups and the way this relates with extensions by these fundamental groups. However, we should also have a notion of simply connected proalgebraic groups, which should play roughly the same role as simply connected groups in the classical sense.

\begin{definition}
A connected proalgebraic group $G$ is {\em simply connected} if $G$ has no nontrivial isogenies, i.e., if for any short exact sequence of proalgebraic groups
\[
\xymatrix{
0 \ar[r] & N \ar[r] & G^{\prime} \ar[r] & G \ar[r] & 0	
}
\]
where $G^{\prime}$ is connected and $N$ is zero-dimensional, then $N = 0$.
\end{definition}

This now leads us to the Local Class Field Theory of Serre. It proceeds by studying the category $\PAGrp_{/k}$ and its (higher) homotopy groups in the case in which $k = \overline{k}$, i.e., when $k$ is algebraically closed.

\section{The Local Class Field Theory of Serre}
Throughout this section we assume that $k$ is algebraically closed, i.e., we assume that $k = \overline{k}$ for some algebraic closure $\overline{k}$ of $k$. A deep theorem of Serre (cf.\@ Theorem 1, Section 2 of \cite{SerreCorps}) then illustrates that for any proalgberaic group $G$ over $k$, there is a simply connected group $\overline{G}$ together with a morphism $f:\overline{G} \to G$ such that $\pi_{0}(G) = \Coker(f)$ and $\pi_1(G) = \Ker(f)$. In this case, we can infer that the functor $\overline{G}$ is exact (Theorem 1, Section 2 of \cite{SerreCorps}), which allows us to deduce that $\pi_{n}^{k}(G) = 0$ for all proalgebraic groups $G$ and all $n \geq 2$. Thus the long exact homotopy sequence applied to the short exact sequence
\[
\xymatrix{
0 \ar[r] & G \ar[r] & H \ar[r] & K \ar[r] & 0
}
\]
gives us the exact sequence of homotopy groups
\[
\xymatrix{
0 \ar[r] & \pi_1(G) \ar[r] & \pi_1(H) \ar[r] & \pi_1(K) \ar[r]^{\delta_0} & \pi_0(G) \ar[r] & \pi_0(H) \ar[r] & \pi_0(K) \ar[r] & 0
}
\]
in this case. From here we can deduce the following Theorem that will, in essence, give us Serre's Local Class Field Theory.
\begin{Theorem}[cf.\@ \cite{SuzYosh}]\label{Theorem: LCFT of Serre}
If $k$ is algebraically closed, then there exists a canonical isomorphism of profinite groups
\[
\pi_1(\mathbf{U}_K) \cong \Gal(K^{\text{Ab}}/K).
\]
\end{Theorem}

To prove this this, we first will collect some lemmas and propositions about the structure present in this situation and how $k$ being algebraically closed influences the theory (which in turn will make the Local Class Field Theory of Hazewinkel more clear as well). In what follows we will use Galois cohomological arguments without apology. For those interested in that theory, we defer the reader to \cite{SerreCorps} for purely Galois cohomological details, and to \cite{BrownCoh} for purely group cohomological details. Instead we will recall basic facts as we need them. For instance, if $A$ is a $G$-module for some Galois group $G$, then we recall that the Tate cohomology with coefficients in $A$ is given by, for all $n \in \Z$ with norm map $N:H_{0}(G,A) \to H^{0}(G,A)$ given by $a \mapsto \sum_{g \in G} ga$,
\[
\widehat{H}^{n}(G,A) := \begin{cases}
H^{n}(G,A) & {\text{if}}\, n \geq 1; \\
\Coker N & {\text{if}}\, n = 0; \\
\Ker N  & {\text{if}}\, n = -1; \\
H_{-(n+1)}(G,A) & {\text{if}}\, n \leq -2.
\end{cases}
\]
In particular, if $L/K$ is a finite Galois extension of $K$, then $L^{\ast}$ is naturally a left $\Gal(L/K)$-module and $\widehat{H}^{q}(\Gal(L/K),L^{\ast}) = 0$ for all $q \in \Z$.
\begin{remark}
Assume that $K$ is a field and that $L$ is a field extension of $K$. We say that $L$ (or $L/K$ if we wish to be particularly explicit) is a {\em {\Galf}} extension of $K$, or simply a {{\Galf}} extension, if and only if $L/K$ is a finite Galois extension. For example, the field $\Q(\sqrt{2})$ is a {{\Galf}} extension of $\Q$.
\end{remark}

Let us now recall/consider an important series of facts from the theory of local fields and Galois cohomology that are useful in our development of Serre's Local Class Field Theory. If $K$ is a complete valued field and if $L/K$ is a {\Galf} extension, then $L$ is also a complete valued field and the group of valuative units in $L$, $U_L$ has an analogous definition to $U_K$. Then there is a short exact sequence of $\Gal(L/K)$-modules
\[
\xymatrix{
0 \ar[r] & U_L \ar[r] & L^{\ast} \ar[r] & \Z\ar[r] & 0
}
\]
where $\Z$ has a trivial $\Gal(L/K)$-module structure (or, if you prefer, the module structure is induced by the augmentation map of rings $\epsilon:\Z[\Gal(L/K)] \to \Z$ given by $n_gg \mapsto n_g$). Since $\widehat{H}^{q}(\Gal(L/K),L^{\ast}) = 0$ for all $q \in \Z$, it then follows from the cohomology long exact sequence that
\[
\widehat{H}^{q}(\Gal(L/K),U_L) \cong \widehat{H}^{q-1}(\Gal(L/K),U_K)
\]
for all $q \in \Z$. In particular, if we set $q = 0$, it follows that the norm map $N_{L/K}:U_L \to U_K$ is surjective and hence the diagram
\[
\xymatrix{
0 \ar[r] & \Ker N_{L/K} \ar[r] & U_L \ar[r]^{N_{L/K}} & U_K \ar[r] & 0	
}
\]
is an exact sequence of Abelian groups.

We proceed with an aim to prove Theorem \ref{Theorem: LCFT of Serre} now by proving some lemmas that will be necessary for proving Proposition \ref{Prop: Surjection from fundamental group  to Galois group} below. The first will be to show that we can give $\mathbf{U}_L$ the structure of a proalgebraic group in the same way that we gave $\mathbf{U}_K$ the structure of a proalgebraic group, and moreover in such a way that the norm map $N_{L/K}:U_{L} \to U_K$ lifts to a map of proalgebraic groups. 

\begin{lemma}\label{Lemma: UL is proalg group and norm map is morphisme}
The group $U_L$ may be given the structure of a proalgebraic group over $k$ in such a way that the norm map $N_{L/K}:U_L \to U_K$ lifts to a morphism $\mathbf{N}_{L/K}:\mathbf{U}_L \to \mathbf{U}_K$ of proalgebraic groups. In particular, the subgroup $\Ker(N_{L/K})$ lifts to a proalgebraic group over $k$.
\end{lemma}
\begin{proof}
The fact that $U_L$ may be made into a proalgebraic group over $k$ follows mutatis mutandis from the construction of $\mathbf{U}_K$ and the fact that $k = \overline{k}$. In particular, if $\ell$ is the residue field of $L$, $\ell/k$ is an algebraic field extension. Thus $\ell = k$ and $\G_{m,\ell}^{(\infty)} = \G_{m,k}^{(\infty)}$. The fact that $N_{L/K}$ may then be lifted follows from the fact that
\[
\mathbf{U}_{K} \cong \mathbf{U}_{K}^{(1)} \times \G_{m,k}^{(\infty)}
\]
and
\[
\mathbf{U}_{L} \cong \mathbf{U}_{L}^{(1)} \times \G_{m,k}^{(\infty)},
\]
so we can define $\mathbf{N}_{L/K}$ by restricting $N_{L/K}$ to $U_{L}^{(1)}$, verifying its image lands in $U_{K}^{(1)}$, and then lifting to $\mathbf{N}_{L/K}^{(1)}:\mathbf{U}_{L}^{(1)} \to \mathbf{U}_{K}^{(1)}$, and then by taking its component on $\G_{m,k}^{(\infty)}$ to be the identity. We then have that $\Ker(\mathbf{N}_{L/K})$ is a closed subscheme of $\mathbf{U}_{L/K}$, and it is routine to verify that
\[
\Ker(\mathbf{N}_{L/K})(k) = \Ker(N_{L/K}).
\]
This completes the proof of the lemma.
\end{proof}

Let us now consider the subgroup
\[
\Ker(N_{L/K})^{0} := \left\langle \frac{\sigma u}{u} \; : \; \sigma \in \Gal(L/K), u \in U_{L} \right\rangle
\]
of $Ker(N_{L/K})$. Note that this is the group of $k$-rational points of the maximal connected subgroup $\Ker(\mathbf{N}_{L/K})_{0}$ of $\Ker(\mathbf{N}_{L/K})$, i.e., 
\[
\Ker(\mathbf{N}_{L/K})^{0}(k) = \Ker(N_{L/K})^{0}.
\]
Then, by Hilbert 90, it follows that
\begin{align*}
\frac{\Ker(N_{L/K})}{\Ker(N_{L/K})^{0}} &\cong \widehat{H}^{-1}(\Gal(L/K),U_{L}) \cong \widehat{H}^{-2}(\Gal(L/K),\Z) \cong H_{1}(\Gal(L/K),\Z) \\
&\cong \frac{\Gal(L/K)}{[\Gal(L/K),\Gal(L/K)]} =: \Gal(L/K)^{\text{Ab}}.
\end{align*}
Thus we have an isomorphism
\[
\Gal(L/K)^{\text{Ab}} \cong \frac{\Ker(\mathbf{N}_{L/K}(k)}{\Ker(\mathbf{N}_{L/K})^{0}(k)}.
\]
Since $\Ker(\mathbf{N}_{L/K})$ is connected, as it is a closed subgroup of a connected group scheme, we also have that the homotopy group
\[
\pi_{0}(\Ker(\mathbf{N}_{L/K})) = \frac{\Ker(\mathbf{N}_{L/K})(k)}{\Ker(\mathbf{N}_{L/K})^0(k)}
\]
is finite. In particular, the calculation above shows that there is an association
\[
\pi_{0}(\Ker(\mathbf{N}_{L/K})) \sim \Gal(L/K)^{\text{Ab}}.
\]
We will use this association, together with the observation that if a Galois field extension $L/K$ is Abelian then $\Gal(L/K) \cong \Gal(L/K)^{\text{Ab}}$, to prove the proposition below.

\begin{proposition}\label{Prop: Surjection from fundamental group  to Galois group}
There is an epimorphism of profinite groups
\[
\theta:\pi_1(\mathbf{U}_{K}) \to \Gal(K^{\text{Ab}}/K).
\]
\end{proposition}
\begin{proof}
Observe that
\[
\Gal(K^{\text{Ab}}/K) = \lim_{\substack{\longleftarrow \\ L/K\, {\text{{\Galf},\, Abelian}}}} \Gal(L/K) \cong \lim_{\substack{\longleftarrow \\ L/K\, {\text{{\Galf},\, Abelian}}}}\pi_0(\Ker(\mathbf{N}_{L/K})).
\]
We now argue locally, i.e., based on each term of the filtration determining the limit above. Fix an Abelian {\Galf} extension $L/K$. From the homotopy long exact sequence we have that there is a natural map
\[
\pi_1^{k}(\mathbf{U}_K) \to \pi_0(\Ker(\mathbf{N}_{L/K}));
\]
since $\pi_0(\mathbf{N}_{L/K}) \sim \Gal(L/K)^{\text{Ab}} = \Gal(L/K)$, we can replace $\pi_0^{k}(\Ker(\mathbf{N}_{L/K}))$ with $\Gal(L/K)$ and obtain a map of profinite groups
\[
\pi_1(\mathbf{U}_K) \to \Gal(L/K);
\]
this map may be shown to be an epimorphism by noting that it is first nonzero, and then that $\pi_1(\mathbf{U}_K)$ is infinite and $\Gal(L/K) = \Gal(L/K)$ is finite. Taking the limit in $L$ then allows us to conclude that there is an epimorphism
\[
\pi_1(\mathbf{U}_K) \to \Gal(K^{\text{Ab}}/K),
\]
which establishes the proposition.
\end{proof}
\begin{proposition}\label{Prop: Serre existence thm}
The map $\theta:\pi_1^{k}(\mathbf{U}_K) \to \Gal(K^{\text{Ab}}/K)$ is monic.
\end{proposition}
\begin{proof}[Sketch]
This proof is particularly technical, so we only sketch it here. Ultimately it boils down to showing that for all finite cyclic groups $C_{\ell}$ of prime order $\ell$ with $\gcd(\ell, p) = 1$ that
\[
\theta^{\ast}:\ProfAGrp(\Gal(K^{\text{Ab}}/K),C_{\ell}) \to \ProfAGrp_{/k}(\pi_1^{k}(\mathbf{U}_K),C_{\ell})
\]
is an epimorphism in $\Set$, where we regard $N$ as a profinite group with trivial limit structure. For precise details, we defer the reader to section $6$ of \cite{SerreCorps}.
\end{proof}

We finally have all the necessary ingredients for the proof of Serre's local class field theory: Theorem \ref{Theorem: LCFT of Serre}.
\begin{proof}[Proof of Theorem \ref{Theorem: LCFT of Serre}]
The map $\theta:\pi_1^{k}(\mathbf{U}_K) \to \Gal(K^{\text{Ab}}/K)$ is an epimorphism by  Proposition \ref{Prop: Surjection from fundamental group  to Galois group}, and monic by Proposition \ref{Prop: Serre existence thm}. Then, since $\ProfAGrp$ is an Abelian category and hence a regular category, it follows that $\theta$ is an isomorphism and we are done.
\end{proof}

\section{The Local Class Field Theory of Hazewinkel}
From here we develop the LCFT of Hazewinkel. It involves some further nontrivial Galois cohomology and a familiarity with local fields that we take for granted here. However, whenever possible, we will present the theory of ramification and local fields with maximal generality, as, for instance, the splitting of the map $\Gal(\overline{K}/K) \to \Gal(\overline{k}/k)$ does not depend on $K$ and $k$ being of mixed characteristic, but instead on the $p$-cohomological dimension of $\Gal(\overline{k}/k)$ being bounded above by $1$. Much of this theory is a little more delicate than the previous section, as we now have to worry about what happens when the residue field $k$ of the local field $K$ is not algebraically closed (and in particular so we have a nontrivial residual Galois action at play). The theorem that describes Hazewinkel's local class field theory takes this more delicate situation and deduces it from the local class field theory of Serre.

\begin{Theorem}\label{Theorem: LCFT of Hazewinkel}
For a perfect field $k$, there is an isomorphism of proalgebraic groups over $k$
\[
\pi_1^{k}(\mathbf{U}_K) \cong \Iner(K^{\text{Ab}}/K).
\]
\end{Theorem}

We will prove this theorem by reducing it to Serre's Theorem \ref{Theorem: LCFT of Serre} above; this strategy follow the exposition of Suzuki and Yoshida in \cite{SuzYosh}. To do this, however, we will need to recall some Galois-theoretic results about local fields. In particular, we will follow some exercises of \cite{SerreGaloisCoh} that show the short reduction morphism
\[
\Gal(\overline{K}/K) \to \Gal(\overline{k}/k)
\]
splits. The strategy begins by controlling and understanding the tamely ramified extensions of $K$ and then using this to handle the wilder ramified extension fields in smaller steps.
\begin{definition}[\cite{SerreGaloisCoh}]
If $K$ is a discrete valuation field with residue field $k$ of positive characteristic $p > 0$, then we say that a {\Galf} extension $L/K$ is {\em tamely ramified} if the relative inertia group $\Iner(L/K)$ is of order coprime to $p$.
\end{definition}
It is a reasonably routine exercise to show that the following are equivalent for a finite tamely ramified field extension $L/K$. However, in the interest of getting a better understanding of tamely ramified extensions (and what distinguishes them from their wilder cousins), we present a proof of the equality of the different perspectives on tamely ramified fields. In the proposition below, we let $R$ be the ring of integers of $K$ and $S$ be the ring of integers of $L$. The maximal ideal of $R$ will be denoted $\pfrak$, while the maximal ideal of $S$ will be written $\qfrak$. Note that by elementary arguments there is an equality of $S$-ideals
\[
S\pfrak = \qfrak^{e},
\]
where $\pi$ is a uniformizer of $R$ and $e = v_{L}(\pi)$ is the ramification index of $L/K$. Finally, we will write $\ell/k$ to be the  corresponding induced field extension based on the canonical maps $R/\pfrak \to S/\qfrak$.
\begin{proposition}\label{Prop: Equivalent ways of seeing tamely ramified extensions}
The following properties are equivalent:
\begin{enumerate}
	\item The field extension $\ell/k$ is separable and $\gcd(e,p) = 1$;
	\item The order of the inertia group $\Iner(L/K)$ is coprime to $p$;
	\item If $L^{\prime}$ is the maximal subextension $L/L^{\prime}/K$ unramified over $K$, then $\ell/k$ is separable and $p$ does not divide $[L:L^{\prime}]$.
\end{enumerate}
\end{proposition}
\begin{proof}
$(1) \implies (2)$: Recall that the order of the relative inertia group is given by
\[
e[\ell:k]_{\text{insep}}.
\]
Since $\ell/k$ is assumed separable we have $[\ell:k]_{\text{insep}} = 1$ and $\gcd(e,p) = 1$, so
\[
\gcd(e[\ell:k]_{\text{insep}}, p) = \gcd(e,p) = 1.
\]

$(2) \implies (1):$ If $\gcd(e[\ell:k]_{\text{insep}},p) = 1$ then $\gcd(e,p) = 1$ and $\gcd([\ell:k]_{\text{insep}},p) = 1$. However, the characteristic of $k$ is $p$, so any nontrivial inseparable subextension of $\ell/k$ must have degree dividing $p$. Thus there can be no such inseparable subextension and so $\ell/k$ is separable.

$(1) \implies (3):$ In this case write $[L:K] = n$ and note that $[L:K] = [L:L^{\prime}][L^{\prime}:K]$, where $[L^{\prime}:K] = f$ is the unramified degree of the extension. But then it follows that $n = ef$ and $[L^{\prime}:L^{\prime}] = [L:K]/[L^{\prime}:K] = e$, which is what was desired.

$(3) \implies (1):$ Finally, observe that
\[
e = v_{L/K}(\pi)
\]
for some uniformizer $\pi$ of $L^{\prime}$ and $K$ (which may be chosen because $L^{\prime}/K$ is unramified). However,
\[
v_{L/K}(\pi) = v_{L/L^{\prime}}(\pi)v_{L^{\prime}/K}(\pi) = v_{L/L^{\prime}}(\pi) = e.
\]
Moreover, $v_{L/L^{\prime}}(\pi) = [L:L^{\prime}]$ because the extension is  totally ramified and $L^{\prime}$ is the maximal unramified subextension, so this gives that $[L:L^{\prime}] = e$ and hence proves that $\gcd(e,p) = 1$.
\end{proof}

Now define the field, where the colimit is taken in the category $\mathbf{Field}$ of fields,
\[
K^{\text{mod}} := \lim_{\substack{\longrightarrow \\ L/K\, {\text{tamely\, ramified}}}} L
\]
and note that $K \hookrightarrow K^{\text{un}} \hookrightarrow K^{\text{mod}} \hookrightarrow K^{\text{sep}}$. The residue fields of $K^{\text{un}}$ and $K^{\text{mod}}$ are $k^{\text{sep}}$, while that of  $K^{\text{sep}}$ is $\overline{k}$. 

In what proceeds we define the group schemes, for all $n \in \N^{\times}$,
\[
\mu_n := \Spec \frac{k[x]}{(x^n-1)}
\]
so that for any field $\ell/k$, 
\[
\mu_n(\ell) = \mu_n(\Spec \ell) = \lbrace a \in \ell \; | \; a^n = 1\rbrace.
\]
We will use these groups to show that the limit of the $\mu_n(\overline{k})$ along $n$ coprime to $p$ is isomorphic to the Galois group $\Gal(K^{\text{mod}}/K^{\text{un}})$, which will in turn allow us to show that the map $\Gal(K^{\text{mod}}/K) \to \Gal(\overline{k}/k)$ splits.
\begin{lemma}[Exercise 2.a of II.4.3 of \cite{SerreGaloisCoh}]\label{Lemma: Iso of Gal groups}
There is an isomorphism of profinite groups
\[
\Gal(K^{\text{mod}}/K^{\text{un}}) \cong \lim_{\substack{\longleftarrow \\ \gcd(n,p) = 1}} \mu_n(\overline{k}).
\]
\end{lemma}
\begin{proof}
First we observe that unramified extensions are tamely ramified, as $\Iner(L/K) = 1$ for any unramified {\Galf} extension $L/K$, and so in this case $\Iner(L/K) = \mu_1(\overline{k})$. Now, for any $L/K$ tamely ramified of ramification degree $e = v_L(\pi)$, where $\pi$ is a uniformizer of $K$ and $v$ is the valuation on $L$. Consider the short exact sequence of groups
\[
\xymatrix{
	1 \ar[r] & \Iner(L/K) \ar[r] & \Gal(L/K) \ar[r] & \Gal(\ell/k) \ar[r] & 1
}
\]
where $\ell$ is the residue field of $L$. Fix the degree $[\ell:k] = f = [\ell:k]_{\text{sep}}$, and recall that if $[L:K] = n$ then $n = ef$. Let $L/L^{\prime}/K$ be the maximal {\Galf} unramified subextension of $L/K$; consequently the residue fields of $L$ and $L^{\prime}$ are both $\ell$, while $[L^{\prime}:K] = [\ell:k] = f$. Then, from the diagram of groups (with exact rows),
\[
\xymatrix{
1 \ar@{=}[r] & \Iner(L^{\prime}/K)	\ar[r] & \Gal(L^{\prime}/K) \ar[r] & \Gal(\ell/k) \ar[r] & 1 \\
1 \ar[r] & \Iner(L/K) \ar[r] & \Gal(L/K) \ar[r] & \Gal(\ell/k) \ar[r] & 1
}
\]
we get that
\[
\Iner(L/L^{\prime}) \cong \Iner(L/K) \cong \frac{\Gal(L/K)}{\Gal(L^{\prime}/K)} \cong \Gal(L/L^{\prime}).
\]
From the fact that $[L:L^{\prime}] = e$ and from the Eisenstein equation of ramification, it follows that $L/L^{\prime}$ is obtained by adjoining an $e$-th root of $\pi$ to $L^{\prime}$ and hence
\[
\Gal(L/L^{\prime}) \cong \mu_{e}(\ell)
\]
by an argument about ideal factorizations in the local Dedekind domain extension $\CalO_{L}/\CalO_{L^{\prime}}$; however, by base changing to $\overline{k}$ we find also that
\[
\Gal(L/L^{\prime}) \cong \mu_{e}(\ell) \cong \mu_{e}(\overline{k})
\]
by the fact that $L/L^{\prime}$ is Galois. However, this implies then that
\begin{align*}
\Gal(K^{\text{mod}}/K^{\text{un}}) & \cong \Iner(K^{\text{mod}}/K) = \lim_{\substack{\longleftarrow \\ L/K\, {\text{tame}}}} \Iner(L/K) \cong \lim_{\substack{\longleftarrow \\ L/L^{\prime}\, {\text{{\Galf}\, ramified}} \\ {\text{of\, degree}\, (e,p) = 1}}}\Gal(L/L^{\prime}) \\
 & \cong \lim_{\substack{\longleftarrow \\ (n,p) = 1}} \mu_{n}(\overline{k}),
\end{align*}
which proves the lemma.
\end{proof}
This leads us to the next lemma, which is the splitting of $\Gal(K^{\text{mod}}/K) \to \Gal(\overline{k}/k)$; note that it is perhaps more appropriate to write $\Gal(k^{\text{sep}}/k)$ in place of $\Gal(\overline{k}/k)$ in order to match $K^{\text{mod}}$ with its residue field, but because Galois extensions only see separable extensions, we have the equality $\Gal(\overline{k}/k) = \Gal(k^{\text{sep}}/k)$, and so we use what is present in the literature.
\begin{lemma}\label{Lemma: Extension over Kmod spilts}
The extension of groups
\[
\xymatrix{
1 \ar[r] & \Gal(K^{\text{mod}}/K^{\text{un}}) \ar[r] & \Gal(K^{\text{mod}}/K) \ar[r] & \Gal(K^{\text{un}}/K) \ar[r] & 1	
}
\]
splits.
\end{lemma}
\begin{proof}
Begin by choosing a uniformizer $\pi$ of $K$. Now, for all $n \geq 1$ with $n \in \N$ and $\gcd(n,p) =1$, define $\pi_n$ to be a fixed $n$-th root of $p$. Then for any $m,n \geq 1$ with $\gcd(m,p) = \gcd(n,p) = 1$, we have that $\pi_{mn}$ is an $(mn)$-th root and so, since $\pi_{nm}^{m}$ is an $n$-th root of $\pi$, we can arrange the multiplications to be taken so that
\[
(\pi_{mn})^{m} = \pi_{n}.
\]
Now let $H \leq \Gal(K^{\text{mod}}/K)$ be the fix subgroup of the $\pi_n$, i.e., for all $n \in \N$ with $\gcd(n,p) = 1$ and for all $\sigma \in H$, $\sigma(\pi_n) = \pi_n$. We will now prove that $\Gal(K^{\text{mod}}/K^{\text{un}}) \rtimes H \cong \Gal(K^{\text{mod}}/K)$ by showing that $H$ and $\Gal(K^{\text{mod}}/K^{\text{un}})$ have trivial intersection.

First observe that $\id_{K^{\text{mod}}} \in H \cap \Gal(K^{\text{mod}}/K^{\text{un}})$, so let $\sigma \in \Gal(K^{\text{mod}}/K^{\text{un}})$ with $\sigma \ne \id_{K^{\text{mod}}}$. Then there exists an $x \in K^{\text{mod}}$ such that $\sigma(x) \ne x$. Because $K^{\text{mod}}$ takes the form
\[
K^{\text{mod}} = \lim_{\substack{\longrightarrow \\ L/K\, {\text{tame}}}} L
\]
and the colimit is filtered, we can find a tamely ramified extension $L/K$ for which $x$ is represented by an element of $L$ and for which the image of this element under every further embedding $M/L$ remains constant. We will denote this element by $x$ as well. The choice of the field $L$ may be made so that the morphism $\sigma$ is realized minimally by an element in the Galois group $\Gal(L/K)$ in a fashion analogous to the representative of $x$. 

To proceed, let $\pi \in K$ be a uniformizer and let $\omega \in L$ be a uniformizer. Since the field $L$ is tamely ramified, it is a {\Galf} extension of $K$ whose residue field $\ell/k$ is a separable extension with ramification $e = v_{L}(\pi)$ coprime to $p$; moreover, we fix $L/L^{\prime}/K$ the maximal ({\Galf}) unramified subextension of $L$. Using that $L$ is a discrete valuation field, we can write
\[
x = \sum_{n=m}^{\infty} a_n\omega^{n}
\]
for some $m \in \Z$. We also take the $a_n$ to arise as
\[
a_n = \tau(\alpha_n),
\]
where $\tau$ is the Teichm{\"u}ller section $\tau:\ell \to L^{\prime}$. Then we have that
\[
\sum_{n = m}^{\infty} a_n\omega^n = \sum_{n=m}^{\infty} \tau(\alpha_n)\omega^n.
\]
Because $\sigma \in \Gal(K^{\text{mod}}/K^{\text{un}})$, $\sigma$ acts as the identity on $K^{\text{un}}$ and restricts to the identity morphism on $\overline{k}$. However, this implies that as a morphism in $\Gal(L/K)$, $\sigma$ acts as the identity on $\ell$. Applying $\sigma$ to $x$ then implies that
\[
\sigma(x) = \sigma\left(\sum_{n = m}^{\infty} \tau(\alpha_n)\omega^n\right) =\sum_{n = m}^{\infty} \sigma(\tau(\alpha_n))\sigma(\omega)^n = \sum_{m=n}^{\infty} \tau(\tilde{\sigma}(\alpha))\sigma(\omega)^n = \sum_{n = m}^{\infty} \tau(\alpha_n)\sigma(\omega)^n
\]
and, because $\sigma(x) \ne x$, it follows that $\sigma(\omega) \ne \omega$. Note that $\tilde{\sigma}$ is the restriction of $\sigma$ from $\Ocal_{L}$ to $\ell$. As such, $\sigma \notin H$ and consequently $H \cap \Gal(K^{\text{mod}}/K^{\text{un}}) = \lbrace \id_{K^{\text{mod}}}\rbrace$. This allows us to conclude that $G \cong \Gal(K^{\text{mod}}/K^{\text{un}}) \rtimes H$.

We now must prove that $H \cong \Gal(K^{\text{un}}/K)$ in order to complete the proof. Observe that if $\sigma \in H$ is a morphism fixing all the $\pi_n$, then following an argument analogous to the one given above shows that in some properly {\Galf} tamely ramified extension $L/K$, $\sigma$ acts on elements $x$ as
\[
\sigma(x) = \sum_{n = m}^{\infty} \sigma(\tau(\alpha_n))\sigma(\pi_e)^n
\]
where we choose the uniformizer of $L$ to be $\pi_e$. Then, since $\sigma$ fixes all the $\pi_e$, the action above reduces to
\[
\sigma(x) = \sum_{n = m}^{\infty} \sigma(\tau(\alpha_n))\pi_e^n.
\]
Because of this, if $\sigma \ne \id_{K^{\text{mod}}}$, it follows that $\sigma$ is a morphism in $\Gal(K^{\text{un}}/K)$. It may be readily checked that this inclusion $H \to \Gal(K^{\text{un}}/K)$ is a group homomorphism. Now consider the diagram
\[
\xymatrix{
\Gal(K^{\text{mod}}/K^{\text{un}}) \ar[r] & \Gal(K^{\text{mod}}/K^{\text{un}}) \rtimes H \ar[d]^{\cong} \ar[r] & H \ar[d] \ar[r] & 1 \ar[r] & 1 \\
\Gal(K^{\text{mod}}/K^{\text{un}}) \ar[r] \ar@{=}[u] & \Gal(K^{\text{mod}}/K) \ar[r]  & \Gal(K^{\text{un}}/K) \ar[r] & 1 \ar@{=}[u] \ar[r] & 1 \ar@{=}[u]	
}
\]
and note that each row is exact. It follows from the Five Lemma in the category $\Grp$ that the inclusion $H \to \Gal(K^{\text{un}}/K)$ is an isomorphism and so we are done.
\end{proof}
\begin{corollary}
The extension of groups
\[
\xymatrix{
1 \ar[r] & \Gal(K^{\text{mod}}/K^{\text{un}}) \ar[r] & \Gal(K^{\text{mod}}/K) \ar[r] & \Gal(k^{\text{sep}}/k) \ar[r] & 1	
}
\]
splits.
\end{corollary}
\begin{proof}
Use the prior lemma with the isomorphism $\Gal(K^{\text{un}}/K) \cong \Gal(\overline{k}/k) = \Gal(k^{\text{sep}}/k).$
\end{proof}

In order to handle the splitting of the morphism $\Gal(\overline{K}/K) \to \Gal(\overline{k}/k)$, we need to know that the group $\Gal(K^{\text{sep}}/K^{\text{mod}})$ is a pro-$p$-group. We do not prove this here, as it is a technical Galois-theoretic result, but instead refer the reader to Exercise 2(b) of Section 4.3 of \cite{SerreGaloisCoh}. It is a straightforward lemma that uses the fact that the resulting extensions classified by $\Gal(K^{\text{sep}}/K^{\text{mod}})$ arise from inertia groups not coprime to the characteristic of $k$. We will also afterwards recall some facts from the Galois cohomology of fields of positive characteristic.

\begin{lemma}\label{Lemma: Pro-p group gal group}
The group $\Gal(K^{\text{sep}}/K^{\text{mod}})$ is a pro-$p$-group.
\end{lemma}
\begin{definition}[Section I.3.1 of \cite{SerreGaloisCoh}]
If $p \in \Z$ is prime and $G$ is a profinite group, then the {\em $p$-cohomological dimension of $G$}, $\cd_{p}(G)$, is the minimal integer $n$ for which, given every discrete torsion $G$-module $A$ and every $q > n$, the $p$-primary component of $H^{q}(G,A)$ is zero.
\end{definition}
The importance of having $\cd_p(G) \leq 1$ will be seen below, and our proof of the splitting of $\Gal(K^{\text{sep}}/K) \to \Gal(\overline{k}/k)$ will use this crucially. In particular, having this dimension allows one to split all extensions of $G$ by pro-$p$-groups. We give a proposition exploring this, together with the following lifting property: Assume that
\[
\xymatrix{
	1 \ar[r] & P \ar[r] & E \ar[r]^{\pi} & W \ar[r] & 1
}
\]
is an extension of profinite groups. We say $G$ has the lifting property for this extension if for every morphism $f:G \to W$, there exists a morphism $f^{\prime}:G \to E$ making the diagram
\[
\xymatrix{
G \ar[dr]_{f} \ar[r]^{f^{\prime}} & E \ar[d]^{\pi} \\
 & W	
}
\]
commute. In particular, this says that if $E_f$ is the pullback
\[
\xymatrix{
E_f \pullbackcorner \ar[r] \ar[d] & E \ar[d]^{\pi} \\
G \ar[r]_{f} & W
}
\]
then the induced extension
\[
\xymatrix{
1 \ar[r] & P \ar[r] & E_f \ar[r] & G \ar[r] & 1	
}
\]
splits, i.e., there is a section $G \to E_f$ of topological groups.
\begin{proposition}[Section I.3.4 of \cite{SerreGaloisCoh}]\label{Prop: Equiv forms of cdp leq 1}
Let $G$ be a profinite group and $p$ an integer prime. Then the following are equivalent:
\begin{enumerate}
	\item $\cd_p(G) \leq 1$;
	\item Every extension of $G$ by a finite Abelian $p$-group  killed by $p$ splits;
	\item Every extension of $G$ by a pro-$p$-group splits.
\end{enumerate}
\end{proposition}
\begin{proposition}[Proposition 3 of Section II.2.2 of \cite{SerreGaloisCoh}]
If $K$ is a field of positive characteristic $q$, then $\cd_q(K) \leq 1$.
\end{proposition}
\begin{corollary}
If $P$ is a pro-$p$-group and $k$ is a field of characteristic $p$, then any extension of $\Gal(\overline{k}/k)$ by $P$ splits.
\end{corollary}
\begin{corollary}\label{Cor:extension of Kun/K by Ksep/Kmod splits}
Any extension of $\Gal(K^{\text{un}}/K)$ by $\Gal(K^{\text{sep}}/K^{\text{mod}})$ splits.
\end{corollary}
\begin{proof}
Begin by recalling that $\Gal(K^{\text{sep}}/K^{\text{mod}})$ is a pro-$p$-group and that $\Gal(K^{\text{un}}/K) \cong \Gal(k^{\text{sep}}/k) \cong \Gal(\overline{k}/k)$. Since $\cd_p(\Gal(\overline{k}/k)) \leq 1$, the result follows by Proposition \ref{Prop: Equiv forms of cdp leq 1}.
\end{proof}

Let us see now how to use this. As in in the proof of Lemma \ref{Lemma: Extension over Kmod spilts}, we set $H \leq \Gal(K^{\text{mod}}/K)$ to be the subgroup isomorphic to $\Gal(K^{\text{un}}/K)$ which fixes the sequence $(\pi_n)_{(n,p) = 1}$. If we consider the short exact sequence
\[
\xymatrix{
	1 \ar[r] & \Gal(K^{\text{sep}}/K^{\text{mod}}) \ar[r] & \Gal(K^{\text{sep}}/K) \ar[r] & \Gal(K^{\text{mod}}/K) \ar[r] & 1
}
\]
then because $\Gal(K^{\text{mod}}/K) \cong \Gal(K^{\text{mod}}/K^{\text{un}}) \rtimes H$, we can produce a diagram
\[
\xymatrix{
	1 \ar[r] & \Gal(K^{\text{sep}}/K^{\text{mod}}) \ar[r] & \Gal(K^{\text{sep}}/K) \ar[r] & \Gal(K^{\text{mod}}/K) \ar[r] & 1 \\
 & & & H \ar[u]_{i} &
}
\]
where $i:H \to \Gal(K^{\text{mod}}/K)$ is the splitting guaranteed by Lemma \ref{Lemma: Extension over Kmod spilts}. Then we have that the diagram
\[
\xymatrix{
	1 \ar[r] & \Gal(K^{\text{sep}}/K^{\text{mod}}) \ar[r] & \Gal(K^{\text{sep}}/K) \ar[r] & \Gal(K^{\text{mod}}/K) \ar[r] & 1 \\
	1 \ar[r] & \Gal(K^{\text{sep}}/K^{\text{mod}}) \ar@{=}[u] \ar[r] & \Gal(K^{\text{sep}}/K)_{H} \ar[r] \ar[u] & H \ar[u]_{i} \ar[r] & 1
}
\]
commutes, where $\Gal(K^{\text{sep}}/K)_{H}$ is the pullback
\[
\xymatrix{
\Gal(K^{\text{sep}}/K)_{H} \pullbackcorner \ar[r]^{p_1} \ar[d]_{p_2}	& H \ar[d]^{i} \\
\Gal(K^{\text{sep}}/K) \ar[r] & \Gal(K^{\text{mod}}/K)
}
\]
described above. Note that the canonical map $p_2:\Gal(K^{\text{sep}}/K)_{H} \to \Gal(K^{\text{sep}}/K)$ is an embedding of groups by the fact that $i$ is a section and pulling back against monics preserves monics, and by the fact that there is a section $i^{\prime}:H \to \Gal(K^{\text{sep}}/K)_{H}$ by Corollary \ref{Cor:extension of Kun/K by Ksep/Kmod splits}. Thus the map $s_{\text{mod}}:H \to \Gal(K^{\text{sep}}/K)$ defined by the composite
\[
\xymatrix{
H \ar[r]^-{i^{\prime}} \ar[dr]_{s_{\text{mod}}} & \Gal(K^{\text{sep}}/K)_{H} \ar[d]^{p_2} \\
 & \Gal(K^{\text{sep}}/K)	
}
\]
is an embedding of groups and induces a diagram
\[
\xymatrix{
	1 \ar[r] & \Gal(K^{\text{sep}}/K^{\text{mod}}) \ar[r] & \Gal(K^{\text{sep}}/K) \ar[r] & \Gal(K^{\text{mod}}/K) \ar[r] & 1 \\
	1 \ar[r] & \Gal(K^{\text{sep}}/K^{\text{mod}}) \ar@{=}[u] \ar[r] & \Gal(K^{\text{sep}}/K)_{H} \ar[r] \ar[u] & H \ar[u]_{i} \ar[ul]^{s_{\text{mod}}} \ar[r] & 1
}
\]
which commutes in the category of profinite groups. We will use this to prove the proposition below:
\begin{proposition}\label{Prop: Splitting of GalKbar to Galkbar}
The restriction map $\Gal(K^{\text{sep}}/K) \to \Gal(K^{\text{un}}/K)$ splits.
\end{proposition}
\begin{proof}
Consider the short exact sequence
\[
\xymatrix{
1 \ar[r] & \Gal(K^{\text{sep}}/K^{\text{un}}) \ar[r] & \Gal(K^{\text{sep}}/K) \ar[r] & \Gal(K^{\text{un}}/K) \ar[r] & 1	
}
\]
and let $\phi:\Gal(K^{\text{un}}/K) \to H$ be the isomorphism of Lemma \ref{Lemma: Extension over Kmod spilts}. Then since $\Gal(K^{\text{sep}}/K^{\text{mod}})$ is a subgroup of $\Gal(K^{\text{sep}}/K^{\text{un}})$, we can produce the diagram
\[
\xymatrix{
1 \ar[r] & \Gal(K^{\text{sep}}/K^{\text{un}}) \ar[r] & \Gal(K^{\text{sep}}/K) \ar[r]^{r_{\text{un}}} & \Gal(K^{\text{un}}/K) \ar[r] & 1 \\
1 \ar[r] & \Gal(K^{\text{sep}}/K^{\text{mod}}) \ar[r] \ar[u] & \Gal(K^{\text{sep}}/K)_{H} \ar[u]^{p_2} \ar@<.5ex>[r] & H \ar@<.5ex>[l]^-{i^{\prime}} \ar[r] \ar[u]^{\cong}_{\phi^{-1}} \ar[ul]^{s_{\text{mod}}} & 1
}
\]
where the rows are exact, the map $\Gal(K^{\text{sep}}/K^{\text{mod}}) \to \Gal(K^{\text{sep}}/K^{\text{un}})$ is the inclusion, and $s_{\text{mod}}$ the map described above. Define $s_{\text{un}}:\Gal(K^{\text{un}}/K) \to \Gal(K^{\text{sep}}/K)$ by 
\[
s_{\text{un}} := s_{\text{mod}} \circ \phi.
\]
Then since
\[
r_{\text{un}} \circ s_{\text{un}} = r_{\text{un}} \circ s_{\text{mod}} \circ \phi = r_{\text{un}} \circ p_2 \circ i^{\prime} \circ \phi,
\]
and since the restriction $r_{\text{un}}:\Gal(K^{\text{sep}}/K) \to \Gal(K^{\text{un}}/K)$ factors as
\[
\xymatrix{
\Gal(K^{\text{sep}}/K) \ar[rr]^{r_{\text{un}}} \ar[dr]_{r_{\text{mod}}} & & \Gal(K^{\text{un}}/K) \\
 & \Gal(K^{\text{mod}}/K) \ar[ur]_{r_{\text{un}}^{\text{mod}}}	
}
\]
we get further that
\[
r_{\text{un}} \circ p_2 \circ i^{\prime} \circ \phi = r_{\text{un}}^{\text{mod}} \circ r_{\text{mod}} \circ p_2 \circ i^{\prime} \circ \phi.
\]
It is then straightforward to check that
\[
r_{\text{un}}^{\text{mod}} \circ r_{\text{mod}} \circ p_2 \circ i^{\prime} = r_{\text{un}}^{\text{mod}} \circ i \circ p_1 \circ i^{\prime} = r_{\text{un}}^{\text{mod}} \circ i = \phi^{-1}
\]
so
\[
r_{\text{un}}^{\text{mod}} \circ r_{\text{mod}} \circ p_2 \circ i^{\prime} \circ \phi = \phi^{-1} \circ \phi = \id_{\Gal(K^{\text{un}}/K)}.
\]
This proves that $s_{\text{un}}$ is a section of $r_{\text{un}}$ and hence shows that the sequence
\[
\xymatrix{
1 \ar[r] & \Gal(K^{\text{sep}}/K^{\text{un}}) \ar[r] & \Gal(K^{\text{sep}}/K)  \ar[r] & \Gal(K^{\text{un}}/K) \ar[r] & 1	
}
\]
splits.
\end{proof}
\begin{corollary}\label{Cor: splitting of GK to Gk}
The reduction map $\Gal(K^{\text{sep}}/K) \to \Gal(\overline{k}/k)$ splits.
\end{corollary}

We now go back to proving the general class field theory of Hazewinkel. In our ultimate goal of proving Theorem \ref{Theorem: LCFT of Hazewinkel}, we need to know some properties of the proalgebraic group $\mathbf{U}_K$ over $k$ and the profinite group $\pi_1(\mathbf{U}_K)$ when $k$ is not algebraically closed.

Throughout what proceeds, let $L$ be the completion of $K^{\text{un}}$ with respect to the ultarametric topology on $K^{\text{un}}$, i.e., $L = \widehat{K^{\text{un}}}$. It then follows from the limit definition of $\mathbf{U}_K$ that
\[
\mathbf{U}_K(\overline{k}) = \lim_{\longleftarrow} \G_{m,R/\mathfrak{m}^n}^{(\infty)}(\overline{k}) = \lim_{\longleftarrow} \frac{R^{\text{un}}}{\mathfrak{m}_{\text{un}}^n} \cong U_{L}.
\]
Similarly, arguments analogous to those given above imply that base changing $\mathbf{U}_K$ to $\Spec \overline{k}$ gives the isomorphism
\[
\mathbf{U}_K \times_{\Spec k} \Spec \overline{k} \cong \mathbf{U}_L,
\]
and hence gives a $\Gal(\overline{k}/k)$ action on $\mathbf{U}_L$. Explicitly, given an automorphism $\sigma \in \Gal(\overline{k}/k)$, this action is induced from the diagram
\[
\xymatrix{
\mathbf{U}_L \ar[d]_{\tilde{\sigma}} & \ar[l]_-{\cong} \mathbf{U}_K \times_{\Spec k} \Spec \overline{k} \ar[d]^{\id_{\mathbf{U}_K} \times \Spec(\sigma)} \\
\mathbf{U}_L & \mathbf{U}_K \times_{\Spec k} \Spec \overline{k} \ar[l]^-{\cong}
}
\] 
The $\Gal(\overline{k}/k)$ action on $\mathbf{U}_L$ is this action is inherited from the $\Gal(\overline{k}/k)$ action on $\Spec \overline{k}$. This action is compatible with left-derived functors, so the groups $\pi_{i}^{\overline{k}}(\mathbf{U}_L)$ obtain actions of $\Gal(\overline{k}/k)$ as well. This gives a $\Gal(\overline{k}/k)$ action on one side of the isomorphism of Theorem \ref{Theorem: LCFT of Serre}, i.e., on the left-hand side of the isomorphism
\[
\pi_{1}^{\overline{k}}(\mathbf{U}_L) \cong \Iner(L^{\text{Ab}}/L).
\]

On the other side of the isomorphism proposed by Theorem \ref{Theorem: LCFT of Serre}, i.e., on the right-hand side, we note that since the reduction map $\Gal(K^{\text{sep}}/K) \to \Gal(\overline{k}/k)$ admits a section $s:\Gal(\overline{k}/k) \to \Gal(K^{\text{sep}}/K)$, there is a $\Gal(\overline{k}/k)$ action on $\Gal(K^{\text{sep}}/K)$ given by $\sigma\cdot \tau := s(\sigma)\tau$, where $\sigma \in \Gal(\overline{k}/k)$ and $\tau \in \Gal(K^{\text{sep}}/K)$.

In order to use the $\Gal(\overline{k}/k)$-action to prove Theorem \ref{Theorem: LCFT of Hazewinkel} from Theorem \ref{Theorem: LCFT of Serre}, we now need to show how the $\Gal(\overline{k}/k)$-action on $\Gal(K^{\text{sep}}/K)$ reduces to an action on $\Gal(L^{\text{Ab}}/L)$. To see this, note that the action of $\Gal(\overline{k}/k)$ is commutative on $\Gal(K^{\text{sep}}/K)$; consequently, upon taking the Abelianization of $\Gal(K^{\text{sep}}/K)$, we find the action of $\Gal(\overline{k}/k)$ pushes forward along the quotient to an action on $\Gal(K^{\text{sep}}/K)^{\text{Ab}} \cong \Gal(K^{\text{Ab}}/K)$; in particular, this new action is given by following the commuting diagram
\[
\xymatrix{
\Gal(\overline{k}/k) \ar[r] \ar[d]_{s} & \Gal(K^{\text{Ab}}/K)	\\
\Gal(K^{\text{sep}}/K) \ar[r]_{q} & \Gal(K^{\text{sep}}/K)^{\text{Ab}} \ar[u]_{\cong}
}
\]
in $\Grp$. Moreover, tracing through this diagram shows that the action of $\Gal(\overline{k}/k)$ on $\Gal(K^{\text{Ab}}/K)$ takes the following form: First, using the isomorphism $\Gal(\overline{k}/k) \cong \Gal(K^{\text{un}}/K)$, replace $\Gal(\overline{k}/k)$ with $\Gal(K^{\text{un}}/K)$; second, push the action of $\Gal(K^{\text{un}}/K)$ forward by the quotient map to $\Gal(K^{\text{Ab}}/K)$ and note that $\Gal(K^{\text{un}}/K)$ now acts by conjugation because we have quotiented by the derived subgroup. It is then straightforward, but technical, to check that the conjugation action of $\Gal(K^{\text{un}}/K)$ on $\Gal(K^{\text{Ab}}/K)$ restricts to an action on
\[
\Gal(K^{\text{Ab}}/K^{\text{un}});
\]
consequently it follows that there is a $\Gal(K^{\text{un}}/K)$-action on $\Gal(K^{\text{Ab}}/K^{\text{un}})$ which descends along isomorphisms to a $\Gal(\overline{k}/k)$-action on $\Gal(K^{\text{Ab}}/K^{\text{un}})$. Now, because $K^{\text{Ab}} \cong (K^{\text{un}})^{\text{Ab}}$ it follows that
\[
\Gal(K^{\text{Ab}}/K^{\text{un}}) \cong \Gal((K^{\text{un}})^{\text{Ab}}/K^{\text{un}}) \cong \Gal(L^{\text{Ab}}/L)
\]
and hence we have obtained our desired action of $\Gal(\overline{k}/k)$ on $\Gal(L^{\text{Ab}}/L)$. Moreover, the isomorphism
\[
\pi_{1}^{\overline{k}}(\mathbf{U}_L) \cong \Gal(L^{\text{Ab}}/L)
\]
is $\Gal(\overline{k}/k)$-equivariant with respect to these actions, and so we can prove Theorem \ref{Theorem: LCFT of Hazewinkel} by using the category of $\Gal(\overline{k}/k)$-modules. In particular, we provide the followingLemma to illustrate this fact.

\begin{lemma}\label{Lemma: Gal coinvariants are intertia}
There is an isomorphism
\[
\Gal(L^{\text{Ab}}/L)_{\Gal(K^{\text{un}}/K)} \cong \Iner(K^{\text{Ab}}/K).
\]
\end{lemma}
\begin{proof}
We make this calculation in the category of $\Z[\Gal(K^{\text{un}}/K)]$-modules and with $\Gal(K^{\text{Ab}}/K^{\text{un}})$, as $\Gal(K^{\text{Ab}}/K^{\text{un}}) \cong \Gal(L^{\text{Ab}}/L)$. Let
\[
I_{\Gal(K^{\text{un}}/K)} := \langle g - 1 \; | \; g \in \Gal(K^{\text{un}}/K) \rangle \trianglelefteq \Z[\Gal(K^{\text{un}}/K)]
\]
be the augmentation (two-sided) ideal of $\Z[\Gal(K^{\text{un}}/K)]$. Consider that
\[
\Gal(K^{\text{Ab}}/K^{\text{un}})_{\Gal(K^{\text{un}}/K)} = \frac{\Gal(K^{\text{Ab}}/K^{\text{un}})}{I_{\Gal(K^{\text{un}}/K)}\Gal(K^{\text{Ab}}/K^{\text{un}})}.
\]
Since the action of $\Z[\Gal(K^{\text{un}}/K)]$ is determined by the section $s:\Gal(K^{\text{un}}/K) \to \Gal(K^{\text{sep}}/K)$ and the quotient $\Gal(K^{\text{sep}}/K) \to \Gal(K^{\text{Ab}}/K)$, we first note that $\Gal(K^{\text{Ab}}/K^{\text{un}})$ arises as the kernel of the restriction $\Gal(K^{\text{Ab}}/K) \to \Gal(K^{\text{un}}/K)$ as in the commuting diagram
\[
\xymatrix{ 
	1 \ar[r] & \Gal(K^{\text{sep}}/K^{\text{un}})  \ar[r] \ar@{->>}[d] & \Gal(K^{\text{sep}}/K) \ar[r] \ar@{->>}[d] & \Gal(K^{\text{un}}/K) \ar@{=}[d] \ar[r] & 1\\
	1 \ar[r] & \Gal(K^{\text{Ab}}/K^{\text{un}}) \ar[r] & \Gal(K^{\text{Ab}}/K) \ar[r] & \Gal(K^{\text{un}}/K) \ar[r] & 1
}
\]
of $\Z[\Gal(K^{\text{un}}/K)]$-modules; note that the $\Gal(K^{\text{un}}/K)$-action on $\Gal(K^{\text{sep}}/K^{\text{un}})$ comes from the splitting of $\Gal(K^{\text{sep}}/K) \to  \Gal(K^{\text{un}}/K)$. From this it follows that
\begin{align*}
\Gal(K^{\text{Ab}}/K^{\text{un}}) &\cong \frac{\Gal(K^{\text{Ab}}/K^{\text{un}})}{I_{\Gal(K^{\text{un}}/K)}\Gal(K^{\text{Ab}}/K^{\text{un}})} \cong \Ker\big(\Gal(K^{\text{Ab}}/K) \to \Gal(K^{\text{un}}/K)\big) \\
&\cong \Iner(K^{\text{Ab}}/K)
\end{align*}
and we are done.
\end{proof}
\begin{remark}
Alternatively, we can see that $\Ker(\Gal(K^{\text{Ab}}/K) \to \Gal(K^{\text{un}}/K)) \cong \Iner(K^{\text{Ab}}/K) \cong \Gal(K^{\text{Ab}}/K^{\text{un}})$ by considering the commuting diagram
\[
\xymatrix{
1 \ar[r] \ar@{=}[d] & 1 \ar[r] \ar@{=}[d] & \Gal(K^{\text{Ab}}/K^{\text{un}}) \ar@{-->}[d]_{\exists!} \ar[r] & \Gal(K^{\text{Ab}}/K) \ar@{=}[d] \ar[r] & \Gal(K^{\text{un}}/K) \ar[d]^{\cong} \ar[r] & 1 \ar@{=}[d] \\
1 \ar[r] & 1 \ar[r] & \Iner(K^{\text{Ab}}/K) \ar[r] & \Gal(K^{\text{Ab}}/K) \ar[r] & \Gal(\overline{k}/k) \ar[r] & 1	
}
\]
where the top and bottom rows are exact at each spot and the unique map from the Galois group to the Inertia Group is given from the fact that the relative inertia group is the kernel of the reduction map. Appealing to the Five Lemma after truncating the last (right-most) term gives the desired isomorphism.
\end{remark}

Let us to the left-hand side of the isomorphism proposed by Theorem \ref{Theorem: LCFT of Hazewinkel}, i.e., to calculating the $\Gal(\overline{k}/k)$-coinvariants of $\pi_1^{\overline{k}}(\mathbf{U}_L)$. To see this we first consider that if $G$ is a profinite group, we write $\hat{G}$ to be the Pontryagin dual of $G$. Following the theory developed in the theory of proalgebraic groups, we find that for any proalgebraic group $G$ over $k$, there is are isomorphisms
\[
\widehat{\pi_i^{k}(G)} \cong \Ext^i_k(\pi_0(G),\Q/\Z) \cong \lim_{\longrightarrow} \Ext^i_k(\pi_0(G),\Z/n\Z) \cong \lim_{\longrightarrow} \Ext^i_k(G,\Z/n\Z),
\]
where we regard $\Z/n\Z$ as a constant quasialgebraic group. These, together with the spectral sequences
\[
\lim_{\longrightarrow}\Ext^{p}(k,\Ext^{q}_{\overline{k}}(\mathbf{U}_L,\Z/n\Z)) \Rightarrow \lim_{\longrightarrow}\Ext^{p+q}_{k}(\mathbf{U}_K,\Z/n\Z) \cong \Ext^{p+q}_{k}(\mathbf{U}_K,\Q/\Z),
\]
we are able to calculate the desired coninvariants.
\begin{lemma}
There is an isomoprhism
\[
\pi_{1}^{\overline{k}}(\mathbf{U}_L)_{\Gal(\overline{k}/k)} \cong \pi_1(\mathbf{U}_K).
\]
\end{lemma}
\begin{proof}
Consider first that since the group $\mathbf{U}_L$ is connected and the constant group $\Q/\Z$ is not,
\[
\Ext^{0}_{\overline{k}}(\mathbf{U}_L,\Q/\Z) = \PAGrp_{/\overline{k}}(\mathbf{U}_L,\Q/\Z) = 0.
\]
This implies, from the spectral sequence, that
\begin{align*}
\lim_{\longrightarrow}\Ext^{0}_{k}(\mathbf{U}_K,\Ext^{1}_{\overline{k}}(\mathbf{U}_L,\Z/n\Z)) &\cong \lim_{\longrightarrow}\Ext_{k}^{1}(\mathbf{U}_K,\Z/n\Z) \cong \Ext_k^{1}(\mathbf{U}_K,\Q/\Z) \\ 
&\cong \Ext_k^{1}(\pi_0^{k}(\mathbf{U}_K),\Q/\Z)  \cong \widehat{\pi_1^{k}(\mathbf{U}_K)}.
\end{align*}
Thus it follows that
\[
\pi_{1}^{\overline{k}}(\mathbf{U}_L)_{\Gal(\overline{k}/k)} \cong \pi_1(\mathbf{U}_K)
\]
and so we are done.
\end{proof}

With this finally done, we have shown that the Galois coinvariants of both sides of the isomorphism are the same! This allows us to deduce the Local Class Field Theory of Hazewinkel, which we prove below.
\begin{proof}[Proof of Theorem \ref{Theorem: LCFT of Hazewinkel}]
From Theorem \ref{Theorem: LCFT of Serre} we have that
\[
\pi_{1}^{\overline{k}}(\mathbf{U}_L) \cong \Gal(L^{\text{Ab}}/L).
\]
The various lemmas above show that
\[
\pi_1(\mathbf{U}_K) \cong \pi_{1}^{\overline{k}}(\mathbf{U}_L)_{\Gal(\overline{k}/k)} \cong \Gal(L^{\text{Ab}}/L)_{\Gal(\overline{k}/k)} \cong \Iner(K^{\text{Ab}}/K),
\]
which proves the theorem.
\end{proof}

\section{The Local Langlands Correspondence for $\GL_{1,K}$}
We would like, in this final section of this chapter, to explain how the Local Langlands Correspondence for $\GL_{1,K}$, for a $p$-adic field $K$, reduces to LCFT for the local field $K$ (and in particular the LCFT of Hazewinkel presented above). Before we do this, however, we would like to present the Weil group of the local field $K$, as we will need it later. Intuitively, it is a dense subgroup of the absolute Galois group $\Gal(\overline{K}/K)$ of $K$ which is obtained by pulling back against the residue homomorphism $\Gal(\overline{K}/K) \to \Gal(\overline{k}/k)$ and the dense embedding of $\Z$ (regarded as the group $\langle \operatorname{Frob}\rangle$) into $\Gal(\overline{k}/k)$; note that the reduction homomorphism is given through the restriction map $\Gal(\overline{K}/K) \to \Gal(K^{\text{un}}/K)$ and then post-composing with  the isomorphism $\Gal(K^{\text{un}}/U) \cong \Gal(\overline{k}/k)$. In particular, it is reasonable to think of $W_K$ as the topological subgroup of $\Gal(\overline{K}/K)$ induced by pairs of Galois automorphisms $\sigma$ of $\overline{K}/K$ and integers $n$ such that $\sigma$ restricts to the $n$-th power of a lift $\Phi$ of the Frobenius automorphism on $k$ to $\Gal(K^{\text{un}}/K)$. In particular, the {Weil Group of the field $K$} determines a pullback
\[
\xymatrix{
W_K \pullbackcorner \ar[r] 	\ar[d] & \langle \operatorname{Frob} \rangle \ar[d]\\
\Gal(\overline{K}/K) \ar[r] & \Gal(\overline{k}/k)
}
\]
in $\Grp(\Top)$.

The construction above implies that $W_K$ obtains the subspace topology from $\Gal(\overline{K}/K)$, and the density of the map $\Z\to \Gal(\overline{k}/k)$ implies that $W_K$ is dense in $\Gal(\overline{K}/K)$. It also follows from this definition that since the inertia subgroup $\Iner(\overline{K}/K)$ is the kernel of the map $\Gal(\overline{K}/K) \to \Gal(\overline{k}/k)$, that $\Iner(\overline{K}/K) \hookrightarrow W_K$ is a topological subgroup. In particular, $W_K$ fits into the following diagram
\[
\xymatrix{
0 \ar[r] & \Iner(\overline{K}/K) \ar[r] & \Gal(\overline{K}/K) \ar[r] & \Gal(\overline{k}/k) \ar[r] & 0 \\
0 \ar[r] & \Iner(\overline{K}/K) \ar@{=}[u] \ar[r] & W_K \ar[u] \ar[r] & \langle \operatorname{Frob} \rangle \ar[r] \ar[u] & 0	
}
\]
where both rows are short exact sequences of topological groups. In fact, we can also prove the following useful lemma which says that the Abelianization of $W_K$, the group $W_K^{\text{Ab}} = W_K/[W_K,W_K]$, arises as a pullback of $\langle \Frob \rangle \to \Gal(\overline{k}/k)$ against the map $\Gal(K^{\text{Ab}}/K)$.
\begin{lemma}\label{Lemma: WKAb is pullback against Ab ext}
There is an isomorphism of $W_{K}^{\text{Ab}}$ with the pullback $P = \Gal(K^{\text{Ab}}/K) \times_{\Gal(\overline{k}/k)} \langle \Frob\rangle$ in the category of topological groups.
\end{lemma}
\begin{proof}
Since $W_{K}^{\text{Ab}}$ is the Abelianization of $W_K$, we can realize $W_K^{\text{Ab}}$ as the maximal commutative quotient of $W_K$. However, this implies that we can view $W_K^{\text{Ab}}$ as the subgroup of pairs $(\sigma, n)$ where $\sigma \in \Gal(K^{\text{Ab}}/K)$ and $n \in \Z$; the pairing is then given by identifying $\sigma$ with $\operatorname{Frob}^n$ in the image of the restriction map $\Gal(K^{\text{Ab}}/K) \to \Gal(\overline{k}/k)$. However, this description of $W_K^{\text{Ab}}$ is exactly the one implied by the pullback $P$ above, which implies that $P \cong W_{K}^{\text{Ab}}$, completing the proof of the lemma.
\end{proof}
Using the same reasoning as in the construction of $W_K$, we can see that $W_K^{\text{Ab}}$ is a dense topological subgroup of $\Gal(K^{\text{Ab}}/K)$ containing the inertia group $\Iner(K^{\text{Ab}}/K)$. In particular, the construction of $W_K^{\text{Ab}}$ as a pullback shows that the exactness relations of $W_K^{\text{Ab}}$ follow mutatis mutandis from those that $W_K$ satisfies.
\begin{corollary}
There is a short exact sequence of topological groups of the form:
\[
\xymatrix{
0 \ar[r] & \Iner(K^{\text{Ab}}/K) \ar[r] & \Gal(K^{\text{Ab}}/K) \ar[r] & \Gal(\overline{k}/k) \ar[r] & 0 \\
0 \ar[r] & \Iner(K^{\text{Ab}}/K) \ar[r] \ar@{=}[u] & W_K^{\text{Ab}} \ar[r] \ar[u] & \langle \operatorname{Frob} \rangle \ar[r] \ar[u] & 0
}
\]
\end{corollary}
\begin{proposition}\label{Prop: WKAb is Kast}
	There is an isomorphism of topological groups
	\[
	W_K^{\text{Ab}} \cong K^{\ast}.
	\]
\end{proposition}
\begin{proof}
Begin by recalling the isomorphism $U_K \cong \Iner(K^{\text{Ab}}/K)$. Now consider the short exact sequence of topological groups
\[
\xymatrix{
0 \ar[r] & U_K \ar[r] & K^{\ast} \ar[r] & \Z \ar[r] & 0
}
\]
and note that we can induce a map $K^{\ast} \to W_{K}^{\text{Ab}}$ by sending a fixed uniformizer $\pi$ of $K$ to a fixed lift of Frobenius $\Phi$ from $k$ to $K$. We then produce the following commuting diagram whose rows are short exact sequences:
\[
\xymatrix{
0 \ar[r] & U_K \ar[r] \ar[d]_{\cong} & K^{\ast} \ar[d] \ar[r] & \Z \ar[d]^{\cong} \ar[r] & 0 \\
0 \ar[r] & \Iner(K^{\text{Ab}}/K) \ar[r] & W_{K}^{\text{Ab}} \ar[r] & \langle \operatorname{Frob} \rangle \ar[r] & 0	
}
\]
Applying the Short Five Lemma then implies that $K^{\ast} \cong W_{K}^{\text{Ab}}$ and we are done. 
\end{proof}
\begin{corollary}
The isomorphism $K^{\ast} \to W_K^{\text{Ab}}$ is the local Artin reciprocity map.
\end{corollary}
\begin{proof}
This is immediate from the fact that each unique uniformizer $\pi \in K^{\ast}$ maps to a unique lift of Frobenius $\Phi$; cf.\@ \cite{NeukirchANT} for details on the local Artin reciprocity map.
\end{proof}

We now move to discuss the Local Langlands Correspondence for $\GL_{1,K}$, and in doing so show that it is equivalent to the Local Class Field Theory we have developed up to this point. However, we have to show how to connect two pieces of information: One from the representation theoretic side of things, and the other from the geometric side of things.

The object we study on the representation theoretic side of things is simple enough to describe: We want to describe/parametrize equivalence classes of all admissible irreducible representations of the group $\GL_{1,K}(K) = \G_{m,K}(K) = K^{\ast}$ by continuous group homomorphisms of the Weil group $W_K$ of $K$. In order to go about explaining explicitly how this works, we will describe both the (basic) theory of admissible representations of $\GL_{1,K}$, together with a basic lemma on irreducible (complex) representations of Abelian groups, as well as some of the theory of continuous group homomorphisms from $W_K \to \C^{\ast}$. For one familiar with category theory, our notation will look a little strange in this section; however, we use it to match with what is used in the literature regarding the Local Langlands Correspondence.

\begin{definition}
Let $G$ be an algebraic group over $K$. We say that a representation $\rho:G(K) \to \GL(V)$, for $V$ a complex vector space, is {\em smooth} if for all $v \in V$ the stabilizer (isotropy) subgroup
\[
\operatorname{Stab}_G(v) := \lbrace g \in G \; | \; \rho(g)v = v \rbrace
\]
is an open subgroup of $G(K)$. A smooth representation is moreover said to be {\em admissible} if for all compact open subsets $C$ of $G(K)$, the fix subspace
\[
V^{C} := \lbrace v \in V \; | \; \forall\,c \in C, \rho(c)v = v \rbrace
\]
is finite dimensional.
\end{definition}
\begin{lemma}\label{Lemma: Smooth Irrep for AB group is 1 dim}
Let $G$ be a group  and let $\rho:G \to \GL(V)$ be an irreducible representation, for $V$ a vector space over an algebraically closed field. Then if $G$ is Abelian, $\rho$ is one-dimensional.
\end{lemma}
\begin{proof}
Recall that from Schur's Lemma, because $K$ is algberaically closed, if there exists a $K$-linear endomorphism $T:V \to V$ such that for all $g \in G$, $\rho(g) \circ T = T \circ \rho(g)$, then for all $g \in G$ there is a scalar $\lambda \in K^{\ast}$ such that $\rho(g) = \lambda\id_{V}$, i.e., $\rho(g)$ is a homothety for all $g \in G$. However, since $G$ is Abelian, we have that for all $g,h \in G$ that for all $v \in V$,
\[
(\rho(g) \circ \rho(h))(v) =\rho(g)(\rho(h)v) = \rho(gh)v = \rho(hg)v = \rho(h)(\rho(g)v) = (\rho(h) \circ \rho(g))(v).
\]
But then it follows that we can find, for all $g \in G$, scalars $\lambda \in K^{\ast}$ for which $\rho(g) = \lambda\id_{V}$ and so $\rho$ stabilizes every subspace of $V$. However, since $\rho$ was assumed to be irreducible, it follows that $\dim V = 1$, which proves the lemma.
\end{proof}
\begin{corollary}
If $G$ is a commutative algebraic group over $K$, any admissible irreducible complex representation $\rho:G(K) \to \GL(V)$ is isomorphic to a quasicharacter $\chi:G(K) \to \C^{\ast}$.
\end{corollary}
The corollary above is very important to the Local Langlands Correspondence, as it allows us to work, up to isomorphism of representations, with a much smaller class of representations than we had before. In particular, because any irreducible admissible complex representation of $\GL_{1,K}(K)$ is isomorphic to a $\C^{\ast}$-valued homomorphism of $\GL_{1,K}(K)$, in order to understand the structure of complex admissible irreducible representations of $\GL_{1,K}(K)$, it suffices to understand the group of $\C^{\ast}$-valued homomorphisms of $\GL_{1,K}(K)$. We use this observation to define the class of irreducible admissible representations we will study.
\begin{definition}
We define the set $\mathcal{A}_1(K)$ to be given as
\[
\mathcal{A}_1(K) := \Grp(K^{\ast},\C^{\ast}).
\]
\end{definition}
\begin{remark}
It is not difficult to show that any group homomorphism $K^{\ast} \to \C^{\ast}$ is continuous for the usual topology of $K$ (the one induced by basis $\lbrace 1 +  \mathfrak{m}^{n}K^{\ast} \; | \; n \in \N\rbrace$) when $\C^{\ast}$ has the discrete topology. Thus there is the identification
\[
\mathcal{A}_1(K) \cong \Top(\Grp)(K^{\ast},(\C^{\ast},{\text{Discrete}})),
\]
and it is in this form that we will use $\mathcal{A}_1(K)$ to prove the Local Langlands Correspondence for $\GL_{1,K}$.
\end{remark}

If the purpose of the Local Langlands Programme is to connect the representation-theoretic side of things with the Galois-theoretic/Geometric side of things, we need to understand how to move between the two sides. Explicitly, we need to have an object $\mathcal{G}_1(K)$ with which to compare and contrast $\mathcal{A}_1(K)$.

The object $\mathcal{G}_1(K)$ can be defined by desiring to compare the representation theory of $\GL_{1,K}(K)$ induced by the topological group structure and the Galois-theoretic representation theory. However, we would like to work with a group that always behaves well with respect to the actions on ramification, and in $\Gal(\overline{K}/K)$ it is possible to construct representations which have infinite image induced by the ramification and having a non-finite restriction to Frobenius. But we can get around this --- if we instead restrict to the Weil group, no issue arises, and the density of $W_K \to \Gal(\overline{K}/K)$ allows us to push a continuous representation on $W_K$ forward uniquely to a continuous representation on $\Gal(\overline{K}/K)$. This brings us to define the set $\mathcal{G}_1(K)$.
\begin{definition}
We define the set $\mathcal{G}_1(K)$ to be
\[
\mathcal{G}_1(K) := \Top(\Grp)(W_K, \C^{\ast}).
\]
\end{definition}
\begin{remark}
Note that the set $\mathcal{G}_1(K)$ is the set of Langlands parameters for $\GL_{1,K}$, as the Langlands dual of $\GL_{1,K}$ is $\Gbb_{m,\C}$.
\end{remark}
\begin{remark}
Our notation for the set $\Acal_1(K)$ and $\Gcal_1(K)$ follow the notation and exposition of \cite{VogtLLC}.
\end{remark}
\begin{remark}
Because $\C^{\ast}$ is Abelian, any continuous representation $\rho:W_K \to \C^{\ast}$ factors as
\[
\xymatrix{
W_K \ar[rr]^{\rho} \ar[dr]_{q} & & \C^{\ast} \\	
 & W_K^{\text{Ab}} \ar@{-->}[ur]_{\exists!\widetilde{\rho}}
}
\]
and does so uniquely. Thus we obtain an isomorphism
\[
\mathcal{G}_{1}(K) \cong \Top(\Grp)(W_K^{\text{Ab}},\C^{\ast}).
\]
\end{remark}
The above isomorphism is the final ingredient we need to prove the Theorem which allows us to deduce the Local Langlands Correspondence for $\GL_{1,K}(K)$.

\begin{Theorem}[Local Langlands Correspondence for $\GL_{1,K}$]\label{Cor: LLC for padic GL1}
There is a canonical bijection
\[
\mathcal{A}_1(K) \cong \mathcal{G}_1(K)
\]
which is natural in the sense that if $L/K$ is any algebraic extension of $K$, there is a commuting diagram
\[
\xymatrix{
\Acal_1(K) \ar[r]^{\cong} \ar[d] & \Gcal_1(K) \ar[d] \\
\Acal_1(L) \ar[r]_{\cong} & \Gcal_1(L)
}
\]
\end{Theorem}
\begin{proof}
Let us begin by proving the isomorphism $\Grp(K^{\ast},\C^{\ast}) = \Acal_1(K) \cong \Gcal_1(K) = \Top(\Grp)(W_K^{\text{Ab}},\C^{\ast})$. Since $K^{\ast} \cong W_K^{\text{Ab}}$ by Proposition \ref{Prop: WKAb is Kast}, it suffices to prove that a representation $\pi:K^{\ast} \to \C^{\ast}$ is continuous with respect to the discrete topology on $\C^{\ast}$ if and only if its corresponding $\rho:W_{K}^{\text{Ab}} \to \C^{\ast}$ is continuous in the usual sense. In particular, it suffices to prove that the Artin reciprocity map
\[
\operatorname{Ar}_K:K^{\ast} \to W_K^{\text{Ab}}
\]
gives rise to an isomorphism (via pre-composition) $\Acal_1(K) \cong \Gcal_1(K)$. Note that this in turn amounts to showing that the induced map $\rho$ is locally constant on $W_{K}^{\text{Ab}}$.

To see this, consider that any representation $\pi:K^{\ast} \to \C^{\ast}$ is continuous in the discrete topology on $\C^{\ast}$. Since $W_{K}^{\text{Ab}} \cong K^{\ast}$ is an isomorphism of topological groups by Proposition \ref{Prop: WKAb is Kast}, the corresponding map $\rho$ is discrete continuous, as homeomorphisms are both open and closed bijections. Now, it is straightforward to check that a map $\rho:W_K^{\text{Ab}} \to \C^{\ast}$ is continuous if and only if it is continuous upon restriction to the inertia subgroup $\Iner(K^{\text{Ab}}/K)$. Proceeding from the restricted map $\tilde{\rho}:\Iner(K^{\text{Ab}}/K) \to \C^{\ast}$, we observe that $\tilde{\rho}$ is continuous and $\Iner(K^{\text{Ab}}/K)$ is a compact, totally disconnected space. Thus, by the continuity of $\tilde{\rho}$, it follows that the image of $\tilde{\rho}$ is a compact and totally disconnected subspace of $(\C^{\ast},{\text{Discrete}})$; from the fact that $\C^{\ast}$ in this case has the discrete topology, it follows that the image of $\tilde{\rho}$ is finite. In particular, it follows that $\tilde{\rho}$ is locally constant on $\Iner(K^{\text{Ab}}/K)$, and so is its extension to $\hat{\rho}:W_{K}^{\text{Ab}} \to \C^{\ast}$. However, $\hat{\rho}$ being locally constant is equivalent to $\rho$ being continuous a continuous map, where $\C^{\ast}$ inherits the usual Archimedian topology. Running this same argument in reverse shows the converse direction.
	
We now prove the naturality of the isomorphisms. For any algebraic field extension $L/K$, note that the norm map $N_{L/K}:L^{\ast} \to K^{\ast}$ induces a morphism
\[
\Acal_1(K) = \Grp(K^{\ast},\C^{\ast}) \xrightarrow{N_{L/K}^{\ast}} \Grp(L^{\ast},\C^{\ast}) = \Acal_1(L)
\]
by pre-composition. Similarly, there is a unique morphism $W_L \to W_K$ given essentially via the restriction $\Gal(\overline{L}/L) \to \Gal(\overline{K}/K)$ and noticing this is compatible with the restriction $\Gal(\overline{\ell}/\ell) \to \Gal(\overline{k}/k)$. It is the dashed arrow $\alpha$ in the commuting cube (where the bottom arrows going inwards exist becuase $\overline{L} = \overline{K}$ and $\overline{\ell} = \overline{k}$)
\[
\begin{tikzcd}
 & W_K \ar[dd] \ar[rr] & & \langle \operatorname{Fr}_k \rangle \ar[dd] \\
W_L \ar[dd] \ar[rr, crossing over] \ar[ur, dashed]{}{\exists!\alpha} & & \langle \operatorname{Fr}_{\ell} \rangle  \ar[ur, swap]{}{\cong} \\
 & \Gal(\overline{K}/K) \ar[rr] & & \Gal(\overline{k}/k) \\
\Gal(\overline{L}/L) \ar[ur] \ar[rr] & & \Gal(\overline{\ell}/\ell) \ar[ur] \ar[from = 2-3, to = 4-3, crossing over]
\end{tikzcd}
\]
where $k$ and $\ell$ are the residue fields of $K$ and $L$, respectively, and where both the front and back-most faces of the cube are pullback squares. This impiles that there is a similar map (also called $\alpha$ by abuse of notation) $\alpha:W_L^{\text{Ab}} \to W_K^{\text{Ab}}$. These functions work by taking a pair $(\sigma,n)$ in $W_L$ (or the Abelianization) and sending it to the pair $(\tilde{\sigma},n)$ where
\[
\tilde{\sigma} = \sigma|_{K}
\]
and where
\[
(\operatorname{Fr}_{\ell})|_{k} = \operatorname{Fr}_k.
\]
This gives rise to the map
\[
\Gcal_1(K) = \Top(\Grp)(W_K^{\text{Ab}},\C^{\ast}) \xrightarrow{\alpha^{\ast}} \Top(\Grp)(W_L^{\text{Ab}},\C^{\ast}) = \Gcal_1(L).
\]

For what follows, let $\theta_K:\Acal_1(K) \to \Gcal_1(K)$ and $\theta_L:\Acal_1(L) \to \Gcal_1(L)$ be the isomorphisms constructed in the first half of the proof. To show that the diagram
\[
\xymatrix{
	\Acal_1(K) \ar[r]^{\theta_K} \ar[d]_{N_{L/K}^{\ast}} & \Gcal_1(K) \ar[d]^{\alpha^{\ast}} \\
	\Acal_1(L) \ar[r]_{\theta_L} & \Gcal_1(L)
}
\]
commutes , we will show that for any representation $\rho \in \Acal_1(K)$, both $(\alpha \circ \theta_K)(\rho)$ and $(\theta_L \circ N_{L/K}^{\ast})(\rho)$ have the same action on $W_{L}^{\text{Ab}}$. For this fix some $(\sigma,n) \in W_L^{\text{Ab}}$ and note that on one hand
\[
((\theta_L \circ N_{L/K}^{\ast})(\rho))(\sigma,n) = (\rho \circ N_{L/K} \circ \operatorname{Ar}_L^{-1})(\sigma,n) = (\rho \circ N_{L/K})(x),
\]
where $x \in L^{\ast}$ is the element corresponding to $(\sigma,n)$. This means that if $\pi_L \in L$ is a uniformizer in $L$, $x$ is in an $n$-infinitesimal neighborhood of $\pi_L$ (as $(\sigma,n)$ restricts to the $n$-th power of a lift of Frobenius). Because we only need to consider a single $x \in L^{\ast}$ at a time, we can wolog assume that $x \in F$, where $F$ is a finite intermediate extension $L/F/K$. We then find that
\[
N_{L/K}(x) = \left(\prod_{\tau \in \Gal(F/K(x))} \tau(x)\right)^{[F:K(x)]}
\]
so
\begin{align*}
(\rho \circ N_{L/K})(x) &= \rho\big(N_{L/K}(x)\big) = \rho\left(\left(\prod_{\tau \in \Gal(F/K(x))}\tau(x)\right)^{[F:K(x)]}\right) = \rho\left(\prod_{\tau \in \Gal(F/K(x))}\tau(x)\right)^{[F:K(x)]} \\
&= \prod_{\tau \in \Gal(F/K(x))} \rho\big(\tau(x)\big)^{[F:K(x)]}.
\end{align*}
We now compute the other composition. Note that
\[
\big((\alpha^{\ast} \circ \theta_K)(\rho)\big)(\sigma,n) = \big(\alpha^{\ast} \circ\rho \circ \operatorname{Ar}_K^{-1})(\sigma,n) = (\rho \circ \operatorname{Ar}_K^{-1} \circ \alpha)(\sigma,n) = (\rho \circ \operatorname{Ar}_K^{-1})(\tilde{\sigma},n) = \rho(\tilde{x}).
\]

We now must show that $\rho(\tilde{x})$ is the same as the product above it. Now let $e_x$ and $f_x$ be the ramification and residual index of $K(x)/K$, respectively. Note that this implies that
\[
[K(x):K] = e_xf_x
\]
because each of the fields are themselves local fields. A routine check shows that conditions on the restriction of $W_L^{\text{Ab}} \to W_K^{\text{Ab}}$ being simultaneously compatible with restriction $\Gal(\overline{L}/L) \to \Gal(\overline{K}/K)$ (on the ramified side) and $\langle \operatorname{Fr}_{\ell} \rangle \to \langle \operatorname{Fr}_{k}\rangle$ (on the unramified side) implies that
\begin{align*}
\tilde{x} &= \operatorname{Ar}_K^{-1}(\tilde{\sigma},n) = \operatorname{Ar}_K^{-1}(\alpha(\sigma,n)) = \left(\prod_{\tau \in \Gal(F/K(x))} \tau(x)^{1/e_xf_x}\right)^{[F:K]}
= \left(\prod_{\tau \in \Gal(F/K(x))} \tau(x)^{1/[K(x):K]} \right)^{[F:K]} \\ 
&= \left(\prod_{\tau \in \Gal(F/K(x))}\right)^{[F:K]/[K(x):K]} = \left(\prod_{\tau \in \Gal(F/K(x))} \tau(x) \right)^{[F:K(x)]} = N_{L/K}(x) = N_{L/K}(\alpha(\sigma,n)).
\end{align*}
Because of this it follows that
\[
(\theta_L \circ N_{L/K}^{\ast})(\rho) = (\alpha^{\ast} \circ \theta_K)(\rho),
\]
which completes the proof of the Local Langlands Correspondence for $\GL_{1,K}$.
\end{proof}
\begin{corollary}\label{Theorem: Equiv of Rep Cats gives LLC for GL(1)}
	There is an equivalence of categories
	\[
	\mathbf{SRep}(K^{\ast}) \simeq \mathbf{CRep}(W_K^{\text{Ab}}),
	\]
	where $\mathbf{SRep}(K^{\ast})$ is the category of smooth (complex) representations of $K^{\ast}$ and $\mathbf{CRep}(W_K^{\text{Ab}})$ is the category of continuous (complex) representations of $W_K^{\text{Ab}}$.
\end{corollary}
\begin{proof}
	We begin by recalling that the categories $\mathbf{SRep}(K^{\ast})$ and $\mathbf{CRep}(W_K^{\text{Ab}})$ are semisimple categories. As such, any object $\rho \in \mathbf{SRep}(K^{\ast})_0$ is a direct sum of simple subobjects; however, these simple subobjects are isomorphic to the $1$-dimensional smooth representations of $K^{\ast}$ by Lemma \ref{Lemma: Smooth Irrep for AB group is 1 dim} because $K^{\ast}$ is Abelian. Similarly, any continuous representation $W_K^{\text{Ab}}$ is isomorphic to a direct sum of simple subobjects, which by Lemma \ref{Lemma: Smooth Irrep for AB group is 1 dim} as well corresponds to a $1$-dimensional continuous representation. However, this says that to determine an equivalence of categories as above, we need only establish an equivalence between the categories of isomorphism classes of simple objects of $\mathbf{SRep}(K^{\ast})$ and $\mathbf{CRep}(W_{K}^{\text{Ab}})$. Note that given such a functorial isomorphism, we can extend this to the whole Abelian category $\mathbf{SRep}(K^{\ast})$ (or $\mathbf{CRep}(W_K^{\text{Ab}})$) by extending this assignment linearly through direct sums and then functorially along the isomorphisms.
	
	We now proceed to give a more explicit description of the isomorphism classes of the simple objects of each category. Using Lemma \ref{Lemma: Smooth Irrep for AB group is 1 dim} on $\mathbf{SRep}(K^{\ast})$ gives that
	\[
	\mathbf{SRep}(K^{\ast})_{/\text{iso}}^{\text{simple}} \cong \Grp(K^{\ast},\C^{\ast}) = \Acal_1(K)
	\]
	while using it on $\mathbf{CRep}(W_K^{\text{Ab}})$ gives
	\[
	\mathbf{CRep}(W_K^{\text{Ab}}, \C^{\ast})_{/\text{iso}}^{\text{simple}} \cong \Top(\Grp)(W_K^{\text{Ab}},\C^{\ast}) = \Gcal_1(K).
	\]
	
	We now use the isomorphism of sets $\Acal_1(K) \cong \Gcal_1(K)$ described in Theorem \ref{Cor: LLC for padic GL1}. This gives the desired isomorphism
	\[
	\mathbf{SRep}(K^{\ast})_{/\text{iso}}^{\text{simple}} \cong \mathbf{CRep}(W_K^{\text{Ab}})_{/\text{iso}}^{\text{simple}}
	\]
	which may then be lifted to the desired equivalence $\mathbf{SRep}(K^{\ast}) \simeq \mathbf{CRep}(W_K^{\text{Ab}})$. This completes the proof of the theorem.
\end{proof}

\bibliographystyle{amsplain}
\bibliography{LCFTbib}
\nocite{*}

\end{document}